\documentclass{amsart}%
\usepackage{amsfonts, amsmath, amssymb}%
\usepackage{graphicx, epic, enumerate}
\usepackage{float}

\newtheorem{theorem}{Theorem}
\theoremstyle{definition}

\newtheorem{corollary}[theorem]{Corollary}
\newtheorem{definition}[theorem]{Definition}
\newtheorem{example}[theorem]{Example}
\newtheorem{lemma}[theorem]{Lemma}
\newtheorem{proposition}[theorem]{Proposition}
\newtheorem{remark}[theorem]{Remark}

\numberwithin{equation}{section}

\newcommand{\N}{\mathbb{N}}

\def\co{\colon\thinspace} 

\DeclareMathOperator{\id}{Id}

\begin{document}
\allowdisplaybreaks
\title{On the virtual singular braid monoid}

\author{Carmen Caprau}
\address{Department of Mathematics, California State University, Fresno, CA 93740, USA}
\email{ccaprau@csufresno.edu}
\urladdr{}
\author{Sarah Zepeda}
\email{sarahrosezepeda@gmail.com}

\date{}
\subjclass[2010]{57M25, 57M27; 20F36}
\keywords{braids, monoidal category, Reidemeister-Schreier method, representations of braid structures, virtual singular braids, Yang-Baxter equation}

\begin{abstract} 
We study the algebraic structures of the virtual singular braid monoid, $VSB_n$, and the virtual singular pure braid monoid, $VSP_n$. The monoid $VSB_n$ is the splittable extension of $VSP_n$ by the symmetric group $S_n$. We also construct a representation of $VSB_n$.
\end{abstract}

\maketitle

\section{Introduction}

Virtual singular braids are interesting mathematical objects that are worth studying not only for their algebraic and topological properties but also for their relationship with virtual singular knots and links. This relationship is determined by generalizations of Alexander~\cite{A} and Markov Theorems~\cite{M} for classical braids and links, and it was studied by the authors in~\cite{AS}. The set of isotopy classes of $n$-stranded virtual singular braids (together with the composition of braids given by vertical concatenation) forms a monoid, called the $n$-stranded virtual singular braid monoid and denoted by $VSB_n$. This monoid can be presented in more than one way using generators and relations. The standard presentation uses $4n$ generators, $n$ generators for each type of crossings: classical ($\sigma_i^{\pm 1}$), virtual ($v_i$), and singular ($\tau_i$) crossings. We represent these generators as combinatorial objects describing different types of interactions between consecutive strands in a braid. Among these generators, only $\tau_i$ is not invertible. In addition, the virtual generator $v_i$ can be identified with the transposition $(i, i+1)$ in $S_n$, the symmetric group on $n$ letters.
In~\cite{AS}, it was proved that $VSB_n$ also admits a reduced presentation using fewer generators, namely $\sigma_1^{\pm1}$ and $\tau_1$, together with all of the $v_i$'s. 

There exists a homomorphism from $VSB_n$ onto $S_n$ with kernel the $n$-stranded virtual singular pure braid monoid, $VSP_n$.  The braids $\mu_i = \sigma_i v_i, \, \, \mu_i^{-1} = v_i \sigma_i^{-1}$ and $\gamma_i = \tau_i v_i$ are members of $VSP_n$. The virtual singular pure braids $\mu_i$ and their inverses are particularly interesting due to their intimate connection with the algebraic Yang-Baxter equation. The relationship between the algebraic structure of the virtual pure braid group and the algebraic Yang-Baxter equation was pointed out in ~\cite{BB, BEE, KL2}.

In this paper, we employ the pure braids $\mu_i^{\pm 1}$ and $\gamma_i$---together with the virtual generators $v_i$---to provide another presentation for the $n$-stranded virtual singular braid monoid. The elements $\mu_i^{\pm 1}$ were used in the work by Kauffman and Lambropoulou~\cite{KL2} to give an interpretation of the virtual braid group in terms of a monoidal category. Variants of these braids were also used by Bardakov in~\cite{Ba} to study the structure of the virtual braid group and the virtual pure braid group.  In our work, we refer to $\mu_i^{\pm 1}$ and $\gamma_i$ as the `elementary fusing strings' and represent them as shown in Figure~\ref{fig:connecting-strings}.  We also provide a reduced presentation for the monoid $VSB_n$ with fewer generators, namely $\mu_1^{\pm 1}, \gamma_1$ and $v_i$, for all $1 \leq i \leq n-1$. 

Another goal of this paper is to make use of the Reidemeister-Schreier method (see~\cite{MKS}) to find a set of generators and relations for the virtual singular pure braid monoid. The elementary fusing strings are not sufficient to generate $VSP_n$, and consequently, we define the `generalized fusing strings'. These are elements of $VSP_n$ and can be represented in terms of the elementary fusing strings and the detour move from virtual knot theory. (For a good introduction of virtual knot theory and the detour move in particular, we refer the reader to the work by Kauffman~\cite{K1}.)

We also give an interpretation of $VSB_n$ in terms of a monoidal category, and then use this category to construct a representation of $VSB_n$ into a submonoid of linear operators over $V^{\otimes n}$, where $V$ is a finite dimensional vector space.

The paper is organized as follows. In Section~\ref{sec:introVSBn} we review the standard and the reduced presentations for $VSB_n$, as introduced in~\cite{AS}. In Section~\ref{sec:newpres} we define the elementary fusing strings and find the set of relations among them, yielding to a new presentation for $VSB_n$. In Section~\ref{sec:pure braids} we focus on the virtual singular pure braid monoid, $VSP_n$, and find a presentation for it with generators the generalized fusing strings. We also show that $VSB_n$ is a semi-direct product of $VSP_n$ and  $S_n$. Finally, in Section~\ref{sec:repres} we define the monoidal category $\textbf{FS}$ of fusing strings as the obvious categorification of the monoid $VSB_n$. We then use this category to construct a representation of $VSB_n$ via a monoidal functor  $\textbf{FS} \to \mathbb{K}-\textbf{Vec}$, where $\mathbb{K}-\textbf{Vec}$ is the category of vector spaces over a field $\mathbb{K}$. We close with an example of such a representation.

\section{Virtual singular braids}\label{sec:introVSBn}

In this section, we provide a brief review of the virtual singular braid monoid.
 
Throughout the paper, $n \in \N$, $n \geq 2$. The set of isotopy classes of $n$-stranded virtual singular braids forms a monoid, which we denote by ${VSB}_n$, and call the $n$-stranded virtual singular braid monoid. The monoid operation is the usual braid multiplication: given two $n$-stranded virtual singular braids $\beta$ and $\beta'$, we form the braid $\beta \beta'$ by stacking $\beta$ on top of $\beta'$ and gluing their endpoints. This monoid has as subsets the virtual braid group~\cite{Ka, K1, KL1, KL2, V} and the singular braid monoid~\cite{B, G}. 
The \textit{identity} element of ${VSB}_n$, denoted by $1_n$, is the braid with $n$ vertical strands free of crossings.

 \subsection{Standard presentation for $VSB_n$}
We begin by giving the standard presentation for the virtual singular braid monoid using generators and relations.
\begin{definition} \label{def:vsbn} 
The $n$-stranded \textit{virtual singular braid monoid}, ${VSB}_n$, is the monoid generated by the \textit{elementary virtual singular braids} $\sigma_i,\sigma_i^{-1},v_i$ and $\tau_i$, for $1 \leq i \leq n-1$:

\[  \sigma_i \,\,\, =\,\,\,  \raisebox{-17pt}{\includegraphics[height=.5in]{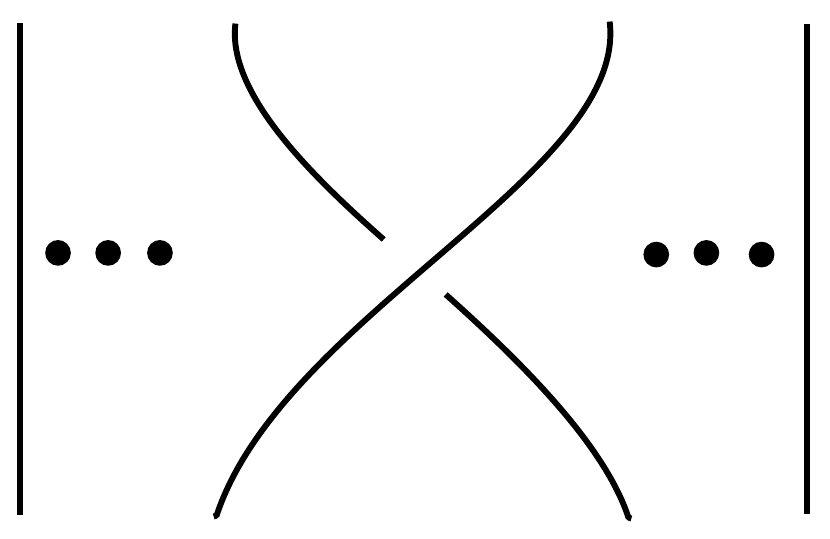}} \hspace{1cm} \sigma_i^{-1} \,\,\,=\,\,\, \raisebox{-17pt}{\includegraphics[height=.5in]{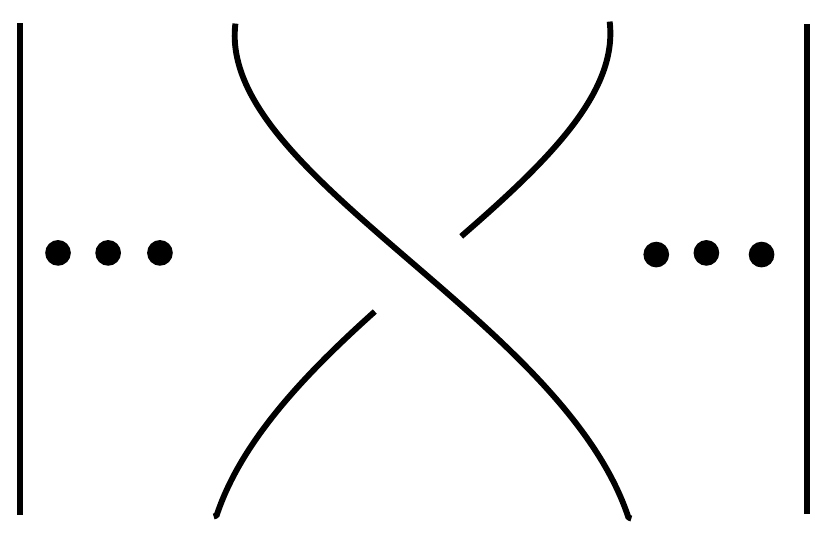}}
\put(-180, 21){\fontsize{7}{7}$1$}
\put(-165, 21){\fontsize{7}{7}$i$}
\put(-150, 21){\fontsize{7}{7}$i+1$}
\put(-127,21){\fontsize{7}{7}$n$}
\put(-58, 21){\fontsize{7}{7}$1$}
\put(-40, 21){\fontsize{7}{7}$i$}
\put(-25, 21){\fontsize{7}{7}$i+1$}
\put(-3,21){\fontsize{7}{7}$n$}
\]
\vspace{0.2cm}
\[ \tau_i \,\,\,= \,\,\, \stackrel \,\, \raisebox{-17pt}{\includegraphics[height=.5in]{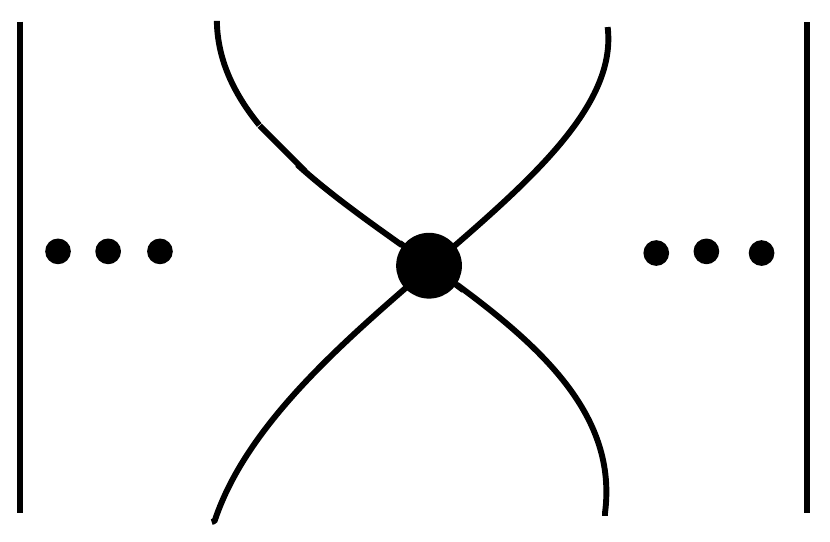}} \hspace{1cm} v_i \,\,\, = \,\,\, \raisebox{-17pt}{\includegraphics[height=.5in]{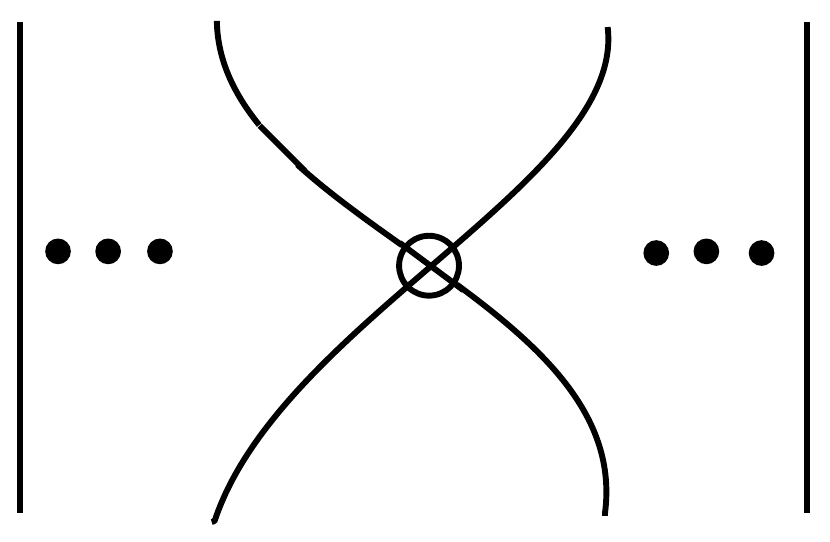}}
\put(-171, 21){\fontsize{7}{7}$1$}
\put(-158, 21){\fontsize{7}{7}$i$}
\put(-140, 21){\fontsize{7}{7}$i+1$}
\put(-119,21){\fontsize{7}{7}$n$}
\put(-56, 21){\fontsize{7}{7}$1$}
\put(-42, 21){\fontsize{7}{7}$i$}
\put(-25, 21){\fontsize{7}{7}$i+1$}
\put(-3,21){\fontsize{7}{7}$n$}
\]
with the defining relations:

\begin{enumerate}
\item $\sigma_i\sigma_i^{-1}=\sigma_i^{-1}\sigma_i=1_n$
\begin{eqnarray*}
\raisebox{-.5cm}{\includegraphics[height=0.7in]{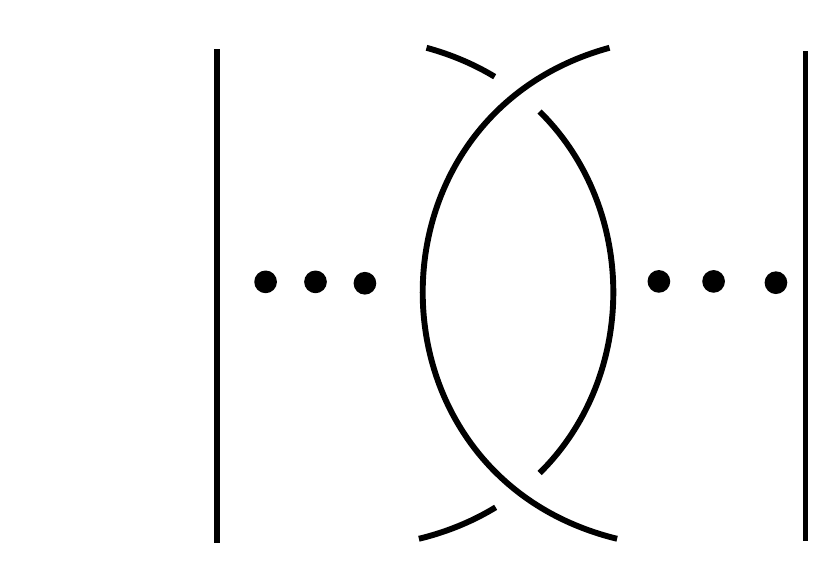}} \hspace{0.5cm} \stackrel{R2}{=} \hspace{0.18cm} &\raisebox{-1cm}{\includegraphics[height=0.85in]{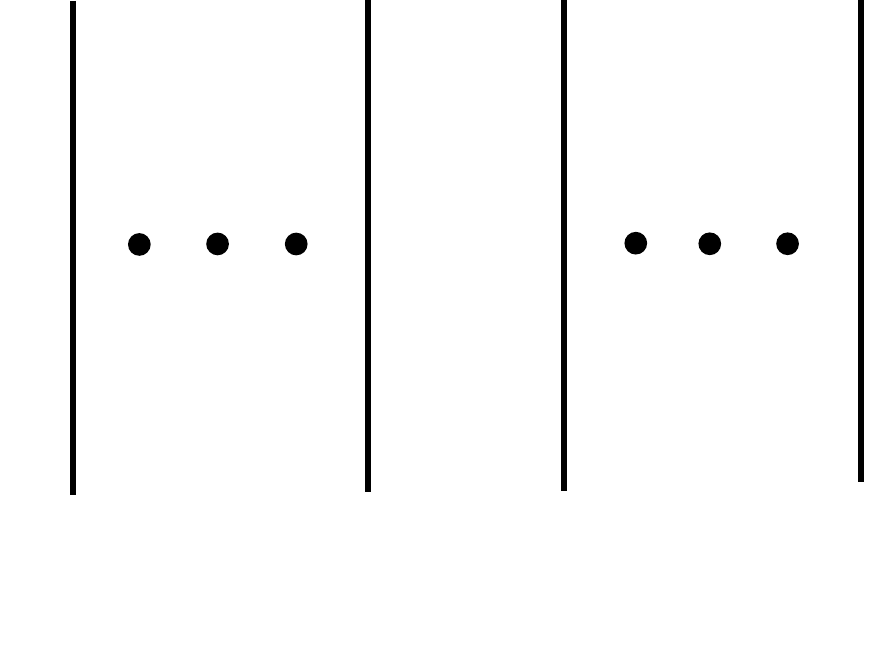}}
\end{eqnarray*}

\item $v_i^2=1_n$
\begin{eqnarray*}
\raisebox{-.5cm}{\includegraphics[height=0.7in]{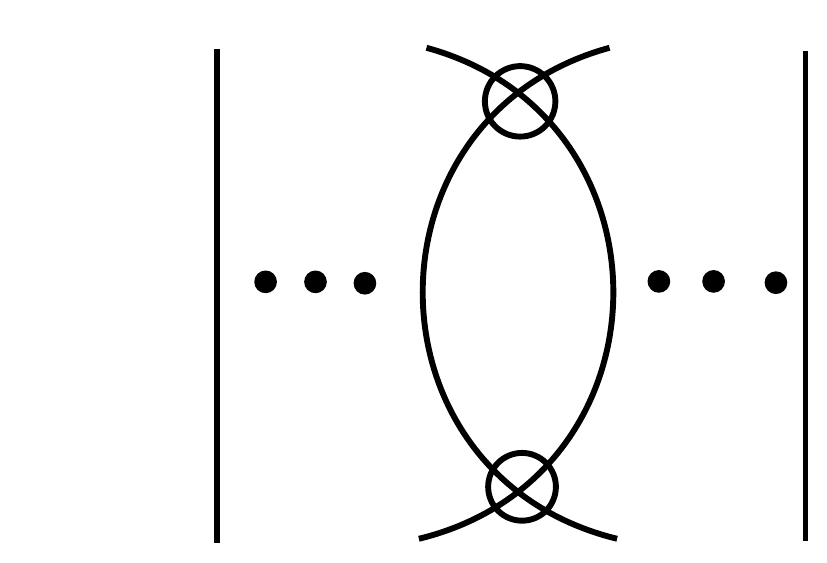}} \hspace{0.5cm} \stackrel{V2}{=} \hspace{0.18cm} &\raisebox{-1cm}{\includegraphics[height=0.85in]{idb}}
\end{eqnarray*}

\item $\sigma_i\sigma_j\sigma_i=\sigma_j\sigma_i\sigma_j$, for $|i-j|=1$
\begin{eqnarray*}
\raisebox{-.5cm}{\includegraphics[height=0.7in]{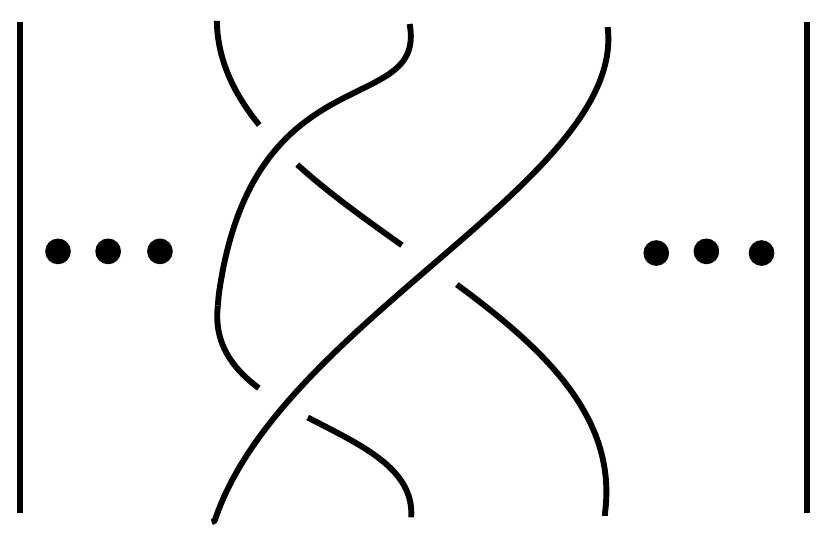}} \hspace{0.5cm} \stackrel{R3}{=} \hspace{0.2cm} &\raisebox{-.5cm}{\includegraphics[height=0.7in]{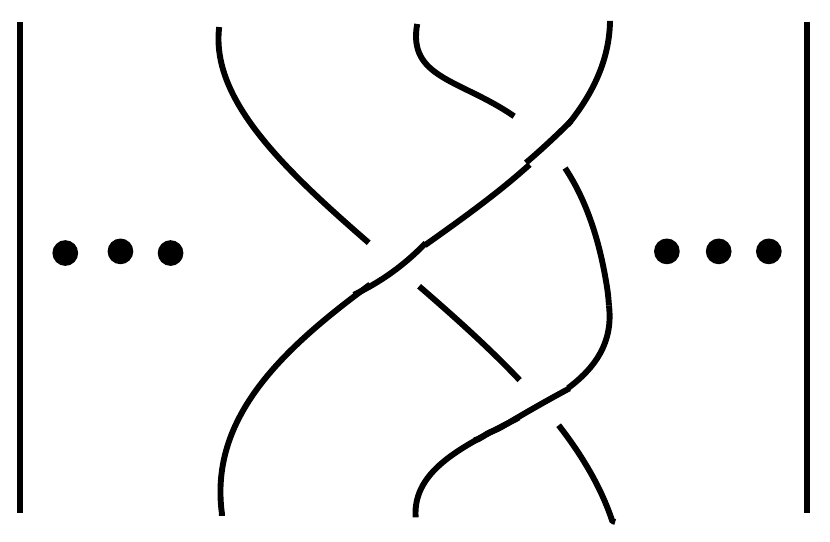}}
\end{eqnarray*}

\item $v_{i}v_{j}v_{i}=v_{j}v_{i}v_{j}$, for $|i-j|=1$
\begin{eqnarray*}
\raisebox{-.5cm}{\includegraphics[height=0.7in]{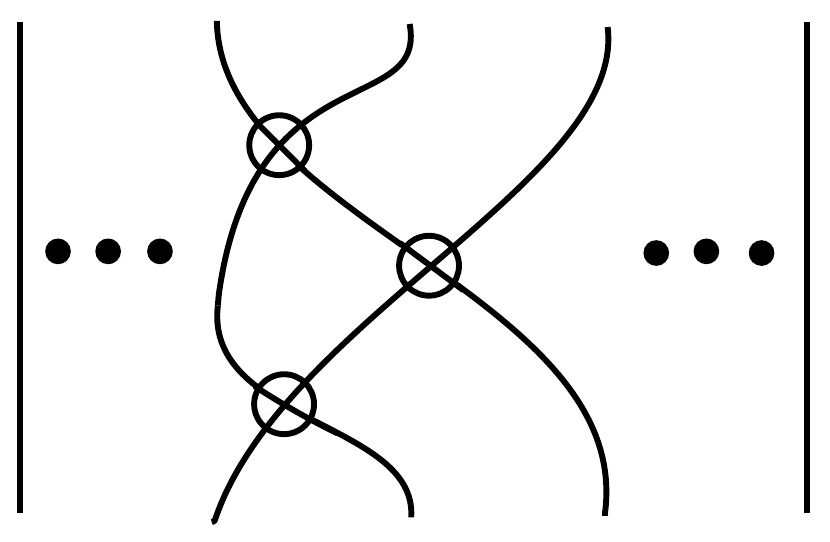}} \hspace{0.5cm} \stackrel{V3}{=} \hspace{0.2cm}&\raisebox{-.5cm}{\includegraphics[height=0.7in]{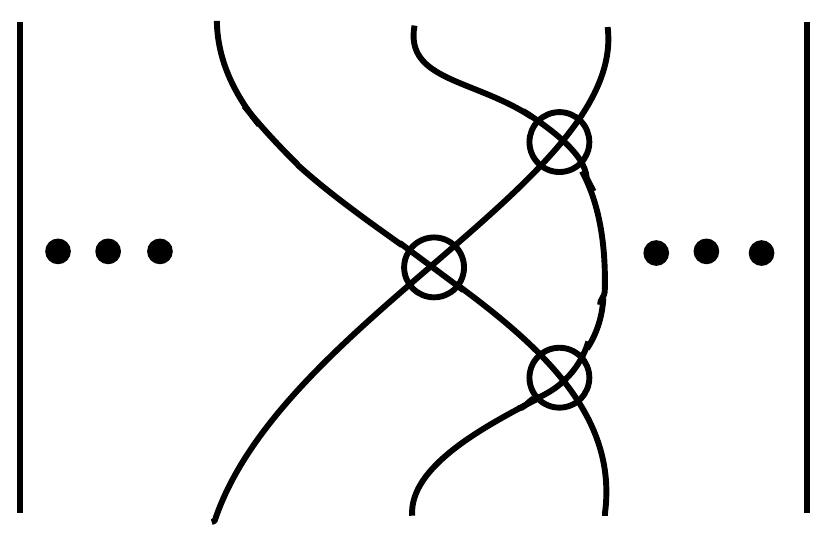}}
\end{eqnarray*}

\item $v_i\sigma_{j}v_i=v_{j}\sigma_iv_{j}$, for $|i-j|=1$ 
\begin{eqnarray*}
\raisebox{-.5cm}{\includegraphics[height=0.7in]{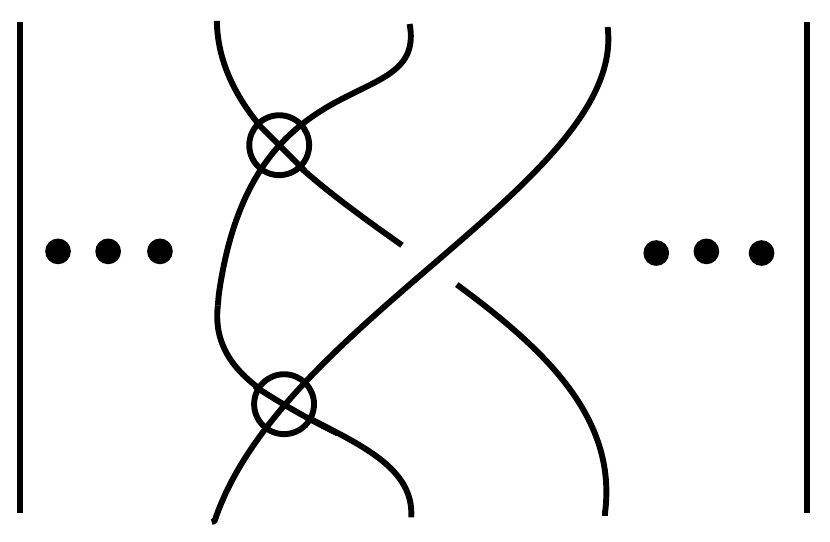}} \hspace{0.5cm} \stackrel{VR3}{=}\hspace{0.2cm} &\raisebox{-.5cm}{\includegraphics[height=0.7in]{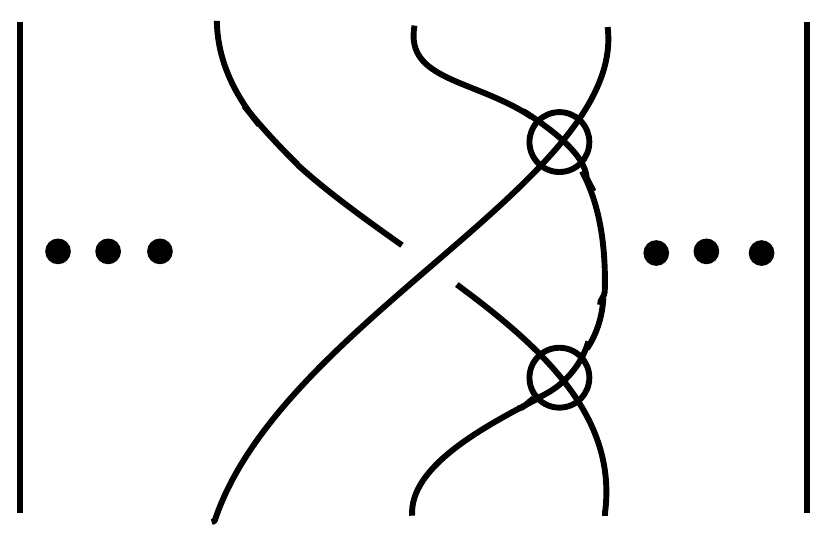}}
\end{eqnarray*}

\item $v_i\tau_{j}v_i=v_{j}\tau_iv_{j}$, for $|i-j|=1$
\begin{eqnarray*}
\raisebox{-.5cm}{\includegraphics[height=0.7in]{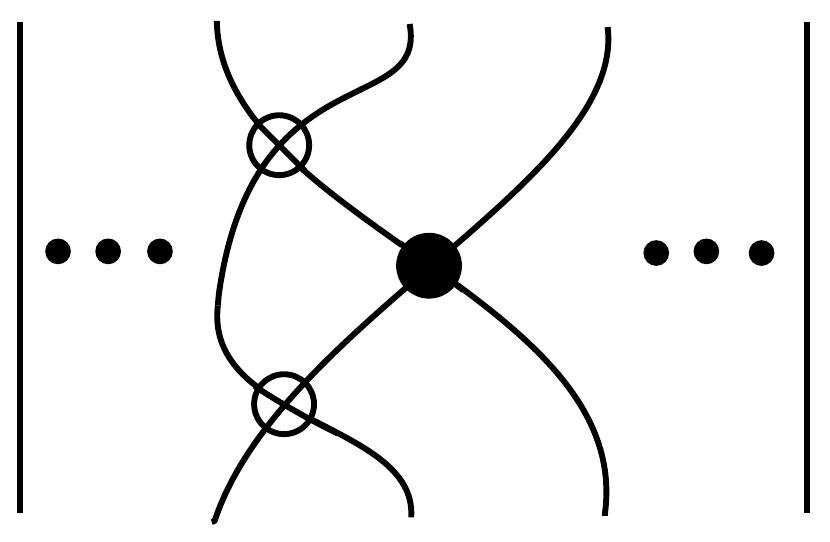}} \hspace{0.5cm} \stackrel{VS3}{=} \hspace{0.2cm} &\raisebox{-.5cm}{\includegraphics[height=0.7in]{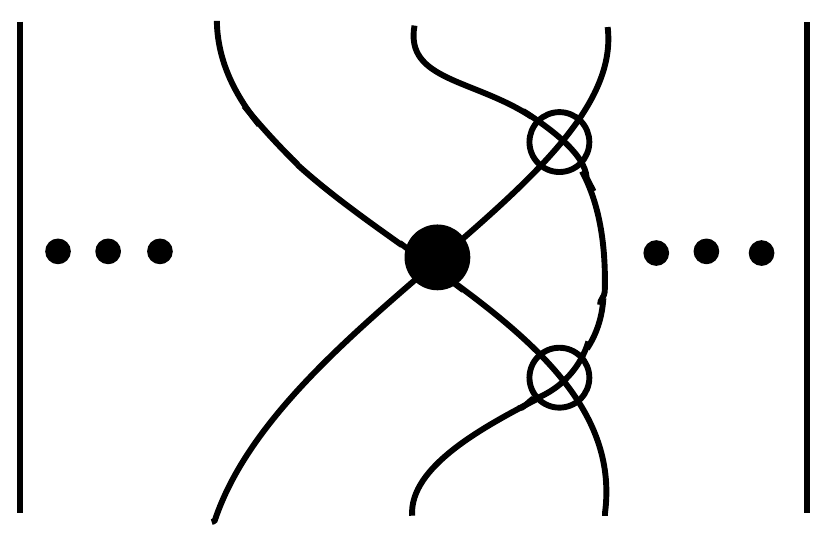}}
\end{eqnarray*}

\item $\sigma_i\sigma_j\tau_i=\tau_j\sigma_i\sigma_j$ for $|i-j|=1$.
\begin{eqnarray*}
\raisebox{-.5cm}{\includegraphics[height=0.7in]{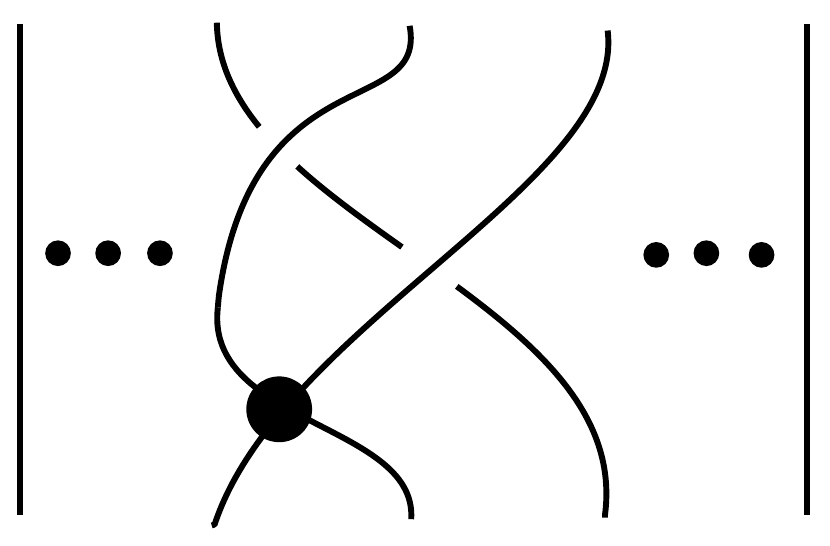}} \hspace{0.5cm} \stackrel{RS3}{=} \hspace{0.2cm} &\raisebox{-.5cm}{\includegraphics[height=0.7in]{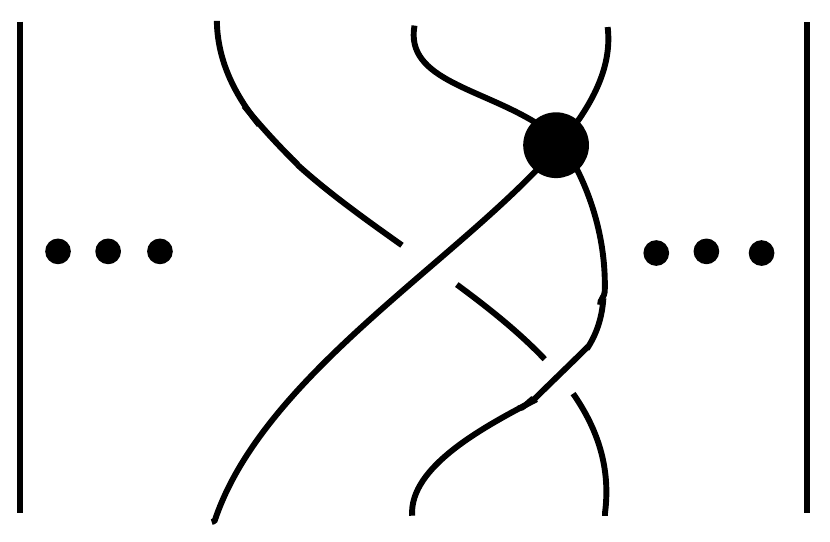}}
\end{eqnarray*}

\item $\sigma_i\tau_i=\tau_i\sigma_i$
\begin{eqnarray*}
\raisebox{-.5cm}{\includegraphics[height=0.7in]{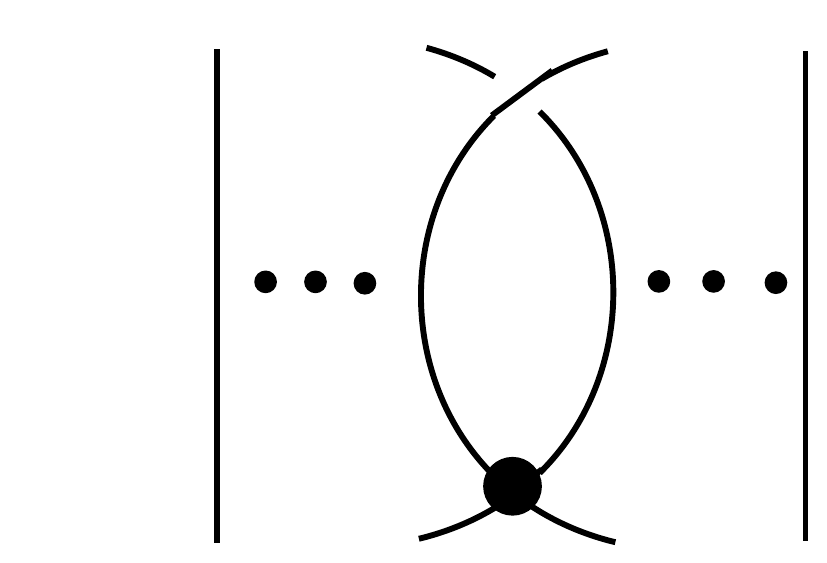}} \hspace{0.5cm} \stackrel{RS1}{=}  \hspace{-0.3cm} &\raisebox{-.5cm}{\includegraphics[height=0.7in]{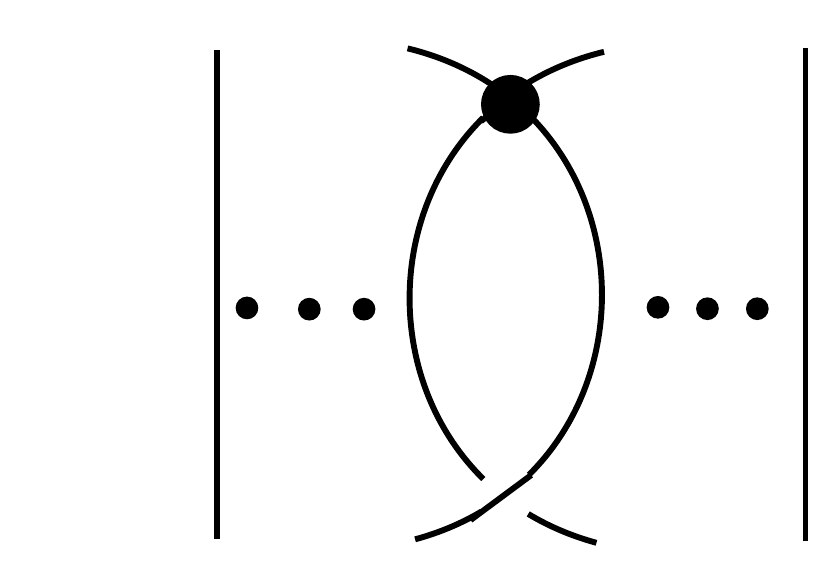}}
\end{eqnarray*}

\item $g_ih_j = h_jg_i $ for $|i-j| >1$, where $g_i, h_i \in \{ \sigma_i, \tau_i, v_i  \}$.
\begin{eqnarray*}
\raisebox{-.5cm}{\includegraphics[height=0.7in]{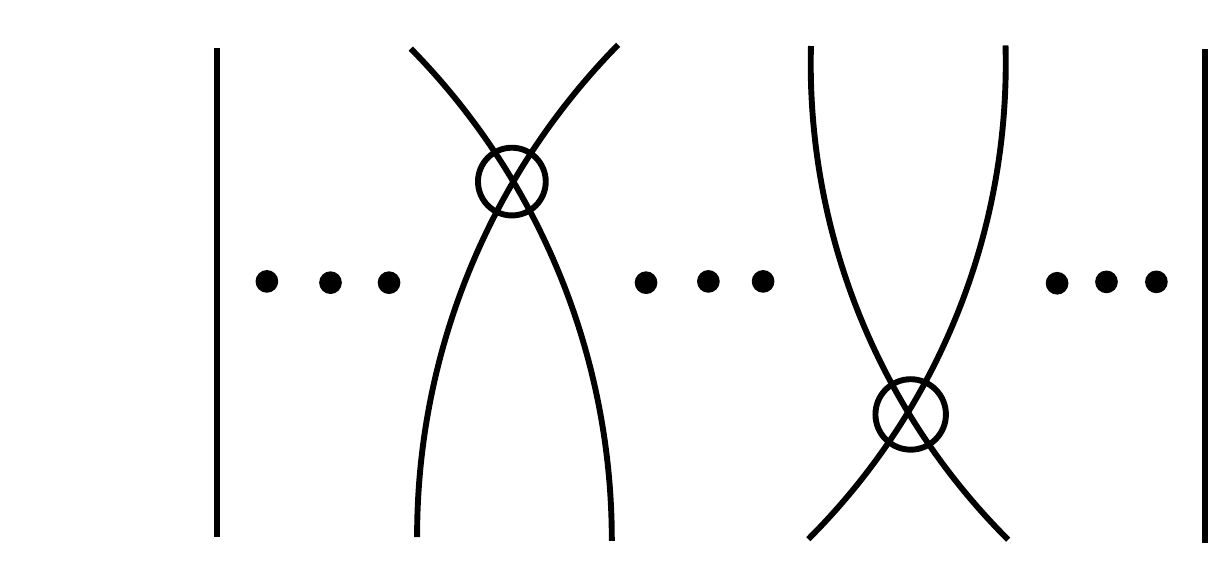}} \hspace{0.5cm} = \hspace{-.3cm} &\raisebox{-.5cm}{\includegraphics[height=0.7in]{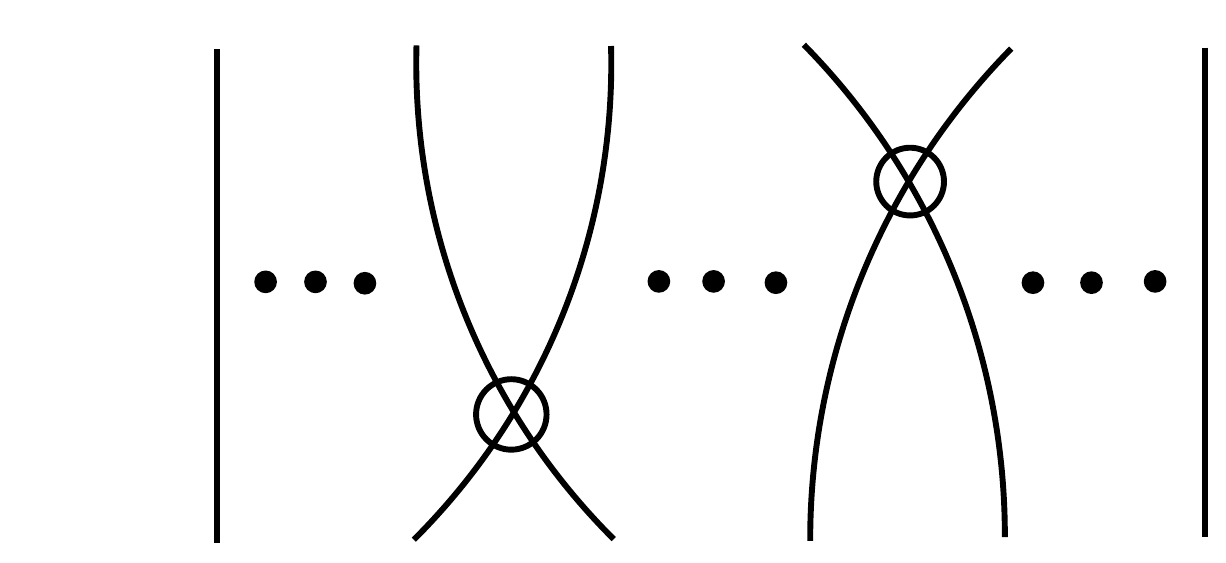}}
\end{eqnarray*}
\end{enumerate}
\end{definition}

These relations taken collectively define the isotopy for virtual singular braids. Each relation in ${VSB}_n$ is a braided version of a Reidemeister-type move for virtual singular link diagrams. We remark that the type I (real and virtual) Reidemeister moves  are not reflected in the defining relations for ${VSB}_n$; these moves cannot be represented using braids, since braid strands must always pass in the downward direction. In addition, the last set of relations listed for $VSB_n$, which we refer to as the \textit{commuting relations}, do not correspond to a move for virtual singular link diagrams. Note that the commuting relations hold for any choice of generators, but we only represented a generic relation using the generators $v_i$.

We refer to the first two relations listed in Definition~\ref{def:vsbn} as the \textit{identity relations}. Likewise, we refer to the third through seventh relations collectively as the \textit{R3-type relations}. Finally, we call the eighth relation the \textit{singular twist relation}.

Note that the generators $\tau_i$ are not invertible in ${VSB}_n$. Naturally, $\sigma_i$ and $\sigma^{-1}_i$ are inverses of each other (by the relation $R2$) and $v_i$ is its own inverse (due to the relation $V2$). 

\begin{remark}
For every $n$-stranded virtual singular braid there is an associated permutation in $S_n$, which describes the permutation of the braid strands (here we ``read" a braid from bottom to top). For example, the braid in Figure~\ref{fig:ex} is associated with the permutation $(143)$ in $S_4$. Moreover, the generators $\sigma_i^{\pm1}, \tau_i$ and $v_i$ for $VSB_n$ are associated with the transpositions $(i, i+1)$ in $S_n$. 

\begin{figure}[ht] 
\[\includegraphics[height=0.8in]{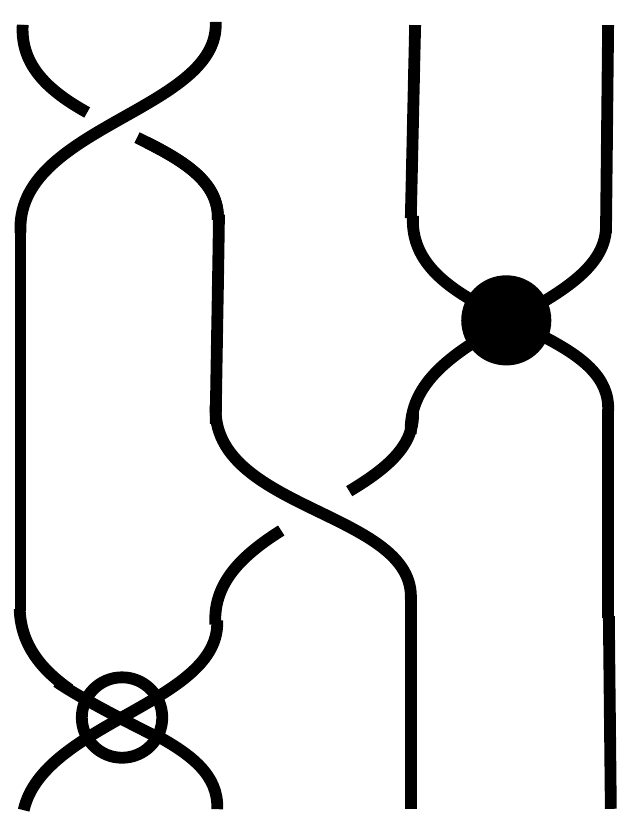}\]
\caption{A braid whose associated permutation is $(143)$}  \label{fig:ex}
\end{figure}
\end{remark}

Therefore, there exists a homomorphism $\pi \co VSB_n\longrightarrow S_n$ defined on the generators of $VSB_n$ by:
 \[\pi(\sigma_i)=\pi(\sigma^{-1}_i)=\pi(v_i)=\pi(\tau_i)=v_i.\]
We use $v_i$ to denote both generators in $VSB_n$ and transpositions $(i, i+1)$ in $S_n$. As elementary virtual singular braids, the generators $v_i$ provide a geometric interpretation of transpositions in $S_n$. Since $S_n$ is generated by transpositions, the homomorphism $\pi$ is surjective.

\begin{definition}
 The $n$-stranded \textit{virtual singular pure braid monoid}, denoted $VSP_n$, is the kernel of the homomorphism $\pi$:
\[ VSP_n: = \ker(\pi), \] 
and an element in the kernel is called a \textit{virtual singular pure braid}. 
\end{definition}
It is clear that the permutation associated with a virtual singular pure braid is the identity permutation in $S_n$. We return to the $n$-stranded virtual singular pure braid monoid in Section~\ref{sec:pure braids}, where we provide a presentation for $VSP_n$ and show that $VSB_n \cong VSP_n  \rtimes S_n$.

\subsection{A reduced presentation for $VSB_n$}
The $R3$-type relations involving virtual crossings are special cases of the \textit{detour move} (see~\cite{K1}), shown in Figure \ref{fig:detour-move}. 
\begin{figure}[ht]
\[ \raisebox{-25pt}{\includegraphics[height=0.8in]{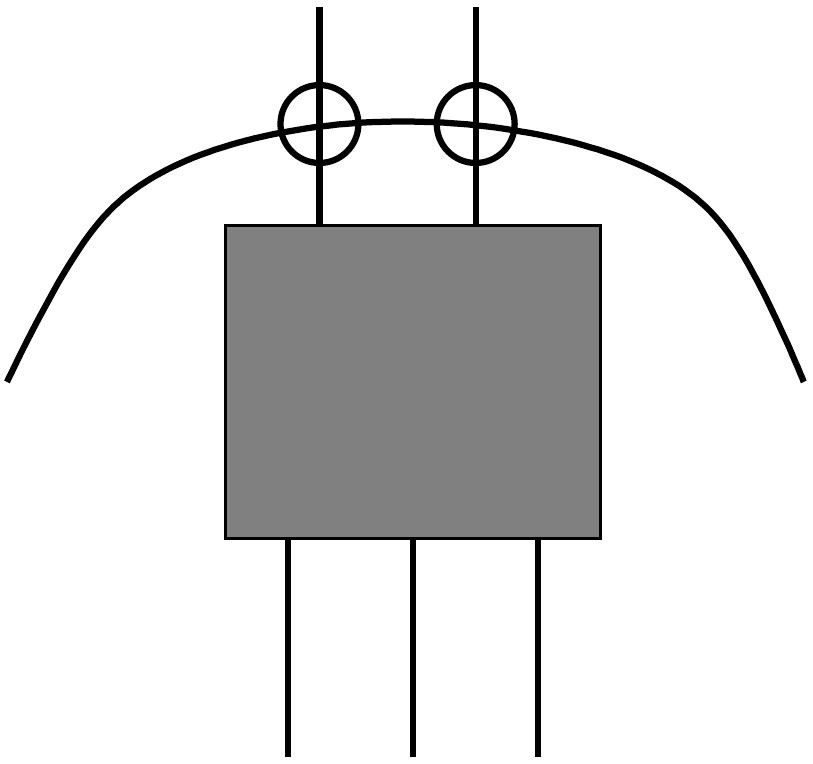}} \hspace{0.3cm}\longleftrightarrow \hspace{0.3cm} \raisebox{-25pt}{\includegraphics[height=0.8in]{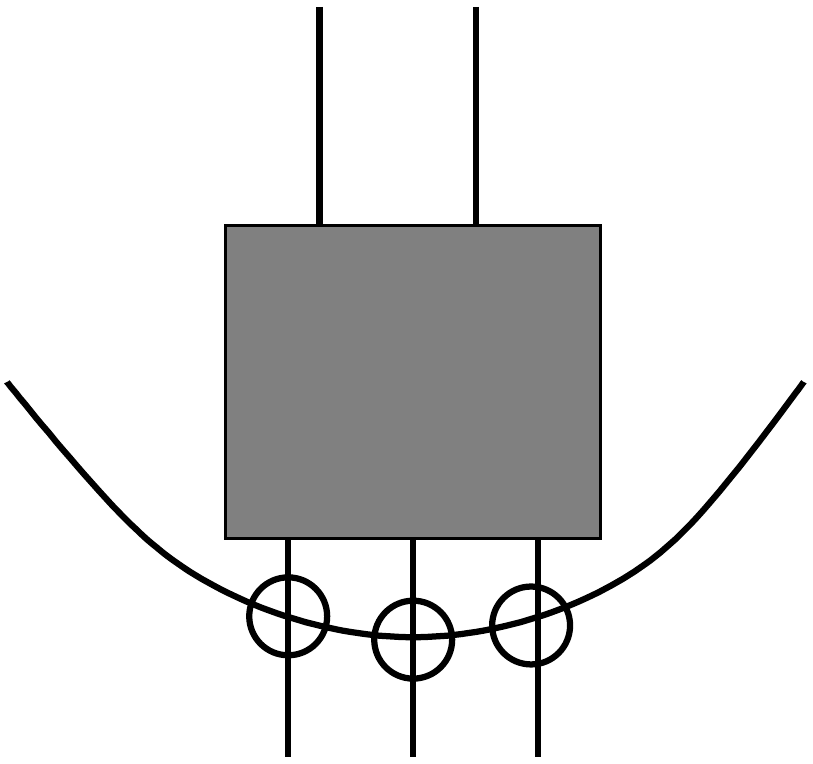}}
\]
\caption{The detour move}
 \label{fig:detour-move}
\end{figure}

The detour move allows us to detour any portion of a braid to the front of the braid, where all of the new crossings (as interactions between the strands of the braids) are virtual. Using the detour move, in~\cite{AS} we derived a reduced presentation for $VSB_n$ using fewer generators. The idea behind this reduced presentation is that the detour move allows us to generate all of the elements of $VSB_n$ using only $\sigma_1, \sigma^{-1}_1, \tau_1$ and the set of all $v_i$, for $1\leq i \leq n-1$. 
Geometrically speaking, we detour the real crossings $\sigma_{i+1}^{\pm1}$ and singular crossings $\tau_{i+1}$ to the left side of the braid using the strands $1, 2, \dots, i$ as shown in Figure~\ref{fig:DRtau}. That is, in algebraic terms, we have:

\begin{eqnarray*}
\sigma_{i+1} ^{\pm 1}  & := & (v_i\ldots v_2v_1)(v_{i+1}\ldots v_3v_2)\sigma_1 ^{ \pm 1} (v_2v_3\ldots v_{i+1})(v_1v_2\ldots v_i)  \label{A14} \\  
\tau_{i+1} &:= & (v_i\ldots v_2v_1)(v_{i+1}\ldots v_3v_2)\tau_1 (v_2v_3\ldots v_{i+1})(v_1v_2\ldots v_i). \label{A15}  
\end{eqnarray*}

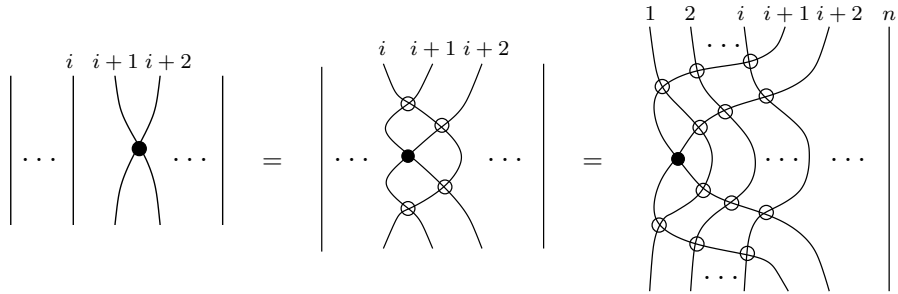
\begin{figure}[ht]
\[ \raisebox{-35pt}{\includegraphics[height=0.8in]{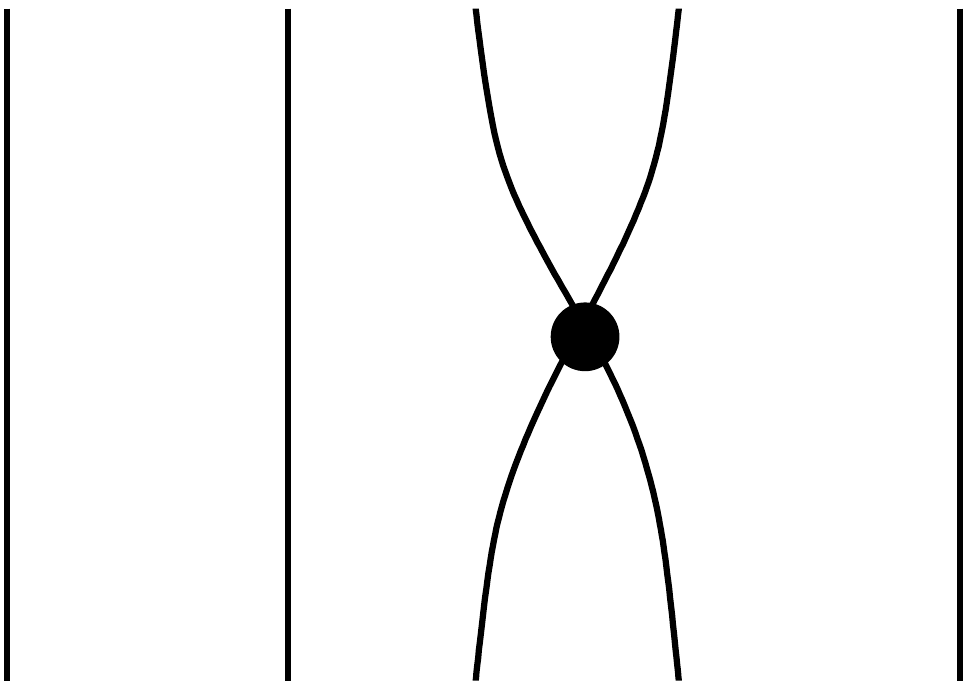}}
 \put(-77, -10){\fontsize{12}{8}$\dots$}
  \put(-20, -10){\fontsize{12}{8}$\dots$}
  \put(-60, 25){\fontsize{8}{8}$i$}
  \put(-50, 25){\fontsize{8}{8}$i+1$}
 \put(-30, 25){\fontsize{8}{8}$i+2$} \hspace{0.5cm}  \raisebox{-13pt}{=}  \hspace{0.5cm} 
\raisebox{-45pt}{\includegraphics[height=1in ]{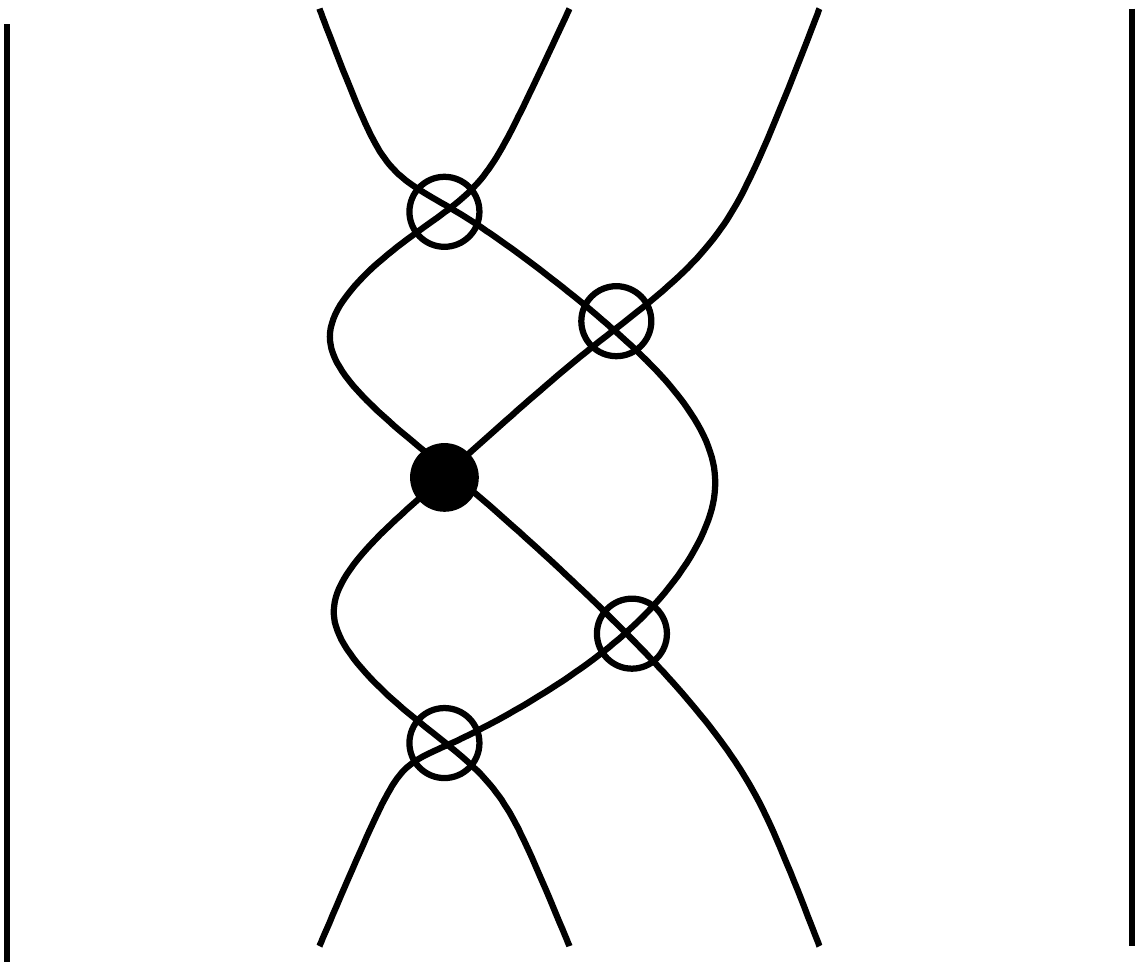}}
 \put(-80, -10){\fontsize{12}{8}$\dots$}
  \put(-22, -10){\fontsize{12}{8}$\dots$}
  \put(-63, 30){\fontsize{8}{8}$i$}
  \put(-51, 30){\fontsize{8}{8}$i+1$}
 \put(-31, 30){\fontsize{8}{8}$i+2$}\hspace{0.5cm} \raisebox{-13pt}{=} \hspace{0.6cm}
 \raisebox{-60pt}{\includegraphics[height=1.4in]{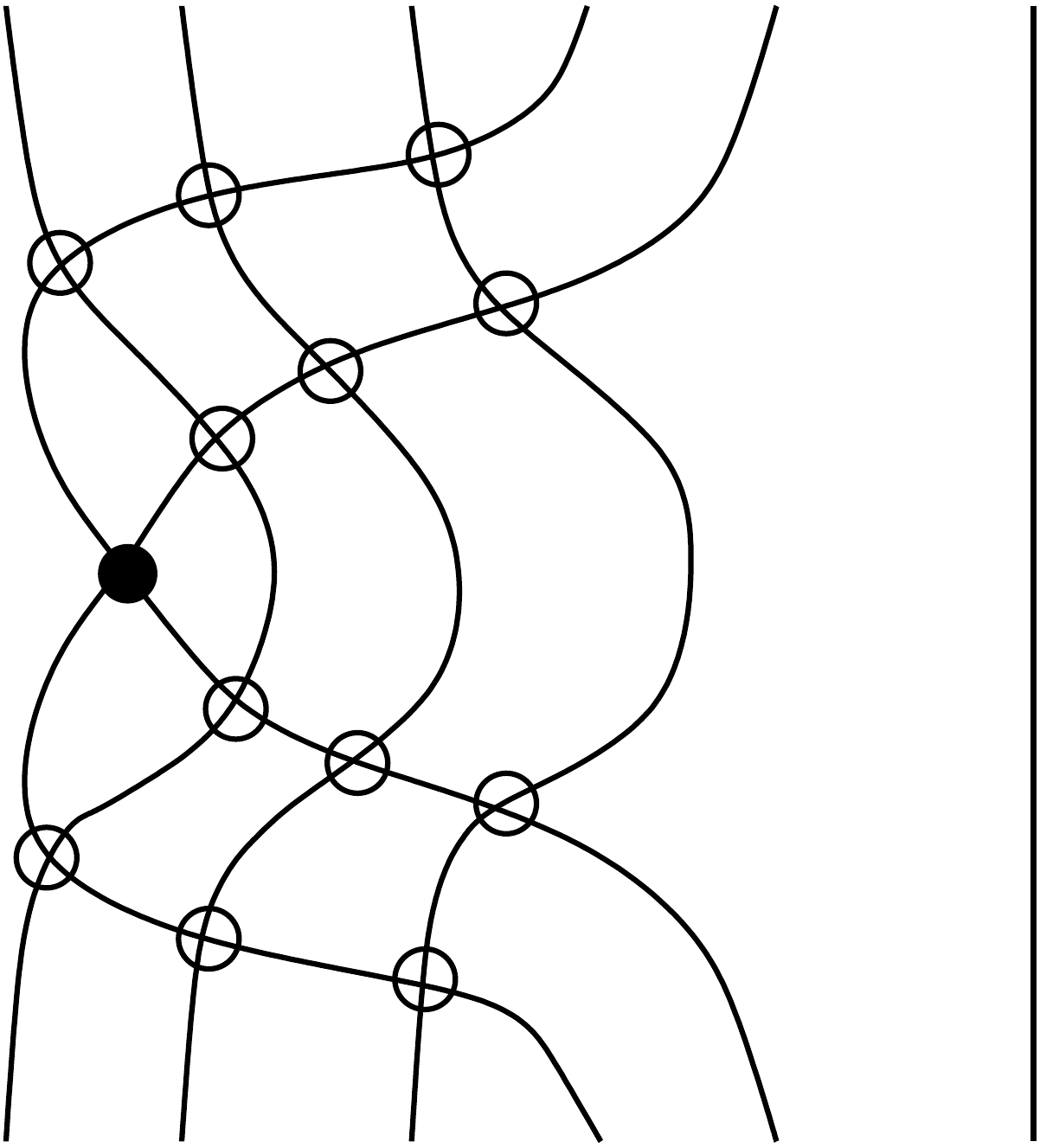}} 
  \put(-23, -10){\fontsize{12}{8}$\dots$}
    \put(-48, -10){\fontsize{12}{8}$\dots$}
     \put(-93, 43){\fontsize{8}{8}$1$}
      \put(-78, 43){\fontsize{8}{8}$2$}
       \put(-71, -55){\fontsize{12}{8}$\dots$}
          \put(-70, 33){\fontsize{12}{8}$\dots$}
    \put(-58, 43){\fontsize{8}{8}$i$}
  \put(-48, 43){\fontsize{8}{8}$i+1$}
 \put(-28, 43){\fontsize{8}{8}$i+2$}
  \put(-3, 43){\fontsize{8}{8}$n$}
 \]
 \caption{Detouring the crossing representing $\tau_{i+1}$}\label{fig:DRtau} 
\end{figure} 

\begin{theorem}[see Theorem 4 in~\cite{AS}]  \label{theorem:red}
$VSB_n$ has the following reduced presentation with generators $\{\sigma_1^{\pm1},\tau_1,v_1,v_2,\dots, v_{n-1}\}$ and relations:
\begin{align}
v_i^2&=1_n \\
       \sigma_1\sigma_1^{-1}&=\sigma_1^{-1}\sigma_1=1_n\\
       \sigma_1\tau_1&=\tau_1\sigma_1\\
       v_iv_jv_i&=v_jv_iv_j,|i-j|=1\\
       \sigma_1 ( v_1 v_2 \sigma_1 v_2 v_1 ) \sigma_1 &= (v_1 v_2 \sigma_1 v_2 v_1 ) \sigma_1(v_1 v_2 \sigma_1 v_2 v_1 )\\
   \tau_1( v_1 v_2 \sigma_1 v_2  v_1 ) \sigma_1  &= ( v_1 v_2 \sigma_1 v_2 v_1 )\sigma_1 (v_1 v_2 \tau_1 v_2 v_1 )\\
\tau_1 v_i = v_i \tau_1\,\,\, &\text{and}\,\, \,\sigma_1v_i=v_i\sigma_1,    \, i \geq 3\\  
              v_i v_j  &= v_j v_i, \,  \,|i -j| >1\\
      \tau_1 (v_2  v_1 v_3 v_2 \tau_1 v_2 v_3 v_1  v_2) &= (v_2  v_1v_3 v_2 \tau_1 v_2 v_3 v_1  v_2) \tau_1 \\
        \tau_1 (v_2  v_1 v_3 v_2 \sigma_1 v_2 v_3 v_1 v_2) &= (v_2  v_1 v_3v_2 \sigma_1 v_2 v_3 v_1  v_2)\tau_1\\
       \sigma_1 (v_2  v_1 v_3 v_2 \sigma_1 v_2 v_3 v_1  v_2)& = (v_2  v_1 v_3 v_2 \sigma_1 v_2 v_3 v_1  v_2)\sigma_1  
\end{align}
\end{theorem}

\section{Virtual singular braid monoid via fusing strings} \label{sec:newpres}

In this section we introduce a new presentation for the $n$-stranded virtual singular braid monoid, $VSB_n$. This presentation uses as generators a special type of pure virtual singular braids, which we now define.   
 
 \subsection{Fusing strings and their relations} \label{sec:FusingStrings}
 
  \begin{definition}\label{def:cstrings}
  The \textit{elementary fusing strings} $\mu_i, \mu_i^{-1}$ and $\gamma_i$, where $1 \leq i \leq n-1$, are $n$-stranded virtual singular braids defined as follows:
   \[\mu_i: = \sigma_iv_i , \hspace{0.5cm} \mu_i^{-1}: = v_i\sigma_i^{-1}, \hspace{0.5cm} \gamma_i: = \tau_iv_i, \,\,\, \text{where} \,\,\,1 \leq i \leq n-1\]
    \end{definition} 
    
    We remark that the elements $\mu_i$ and $\mu^{-1}_i$ are called `connecting strings' in~\cite{KL2}.
    
 Although the elementary fusing strings can be presented algebraically as elements of $VSB_n$, we think of them geometrically as abstract fusions between two adjacent braid strands, as depicted in Figure~\ref{fig:connecting-strings}.  
  \begin{figure}[ht]   \[\mu_i\hspace{.2cm}=\hspace{.2cm}\raisebox{-.7cm}{\includegraphics[height=.6in]{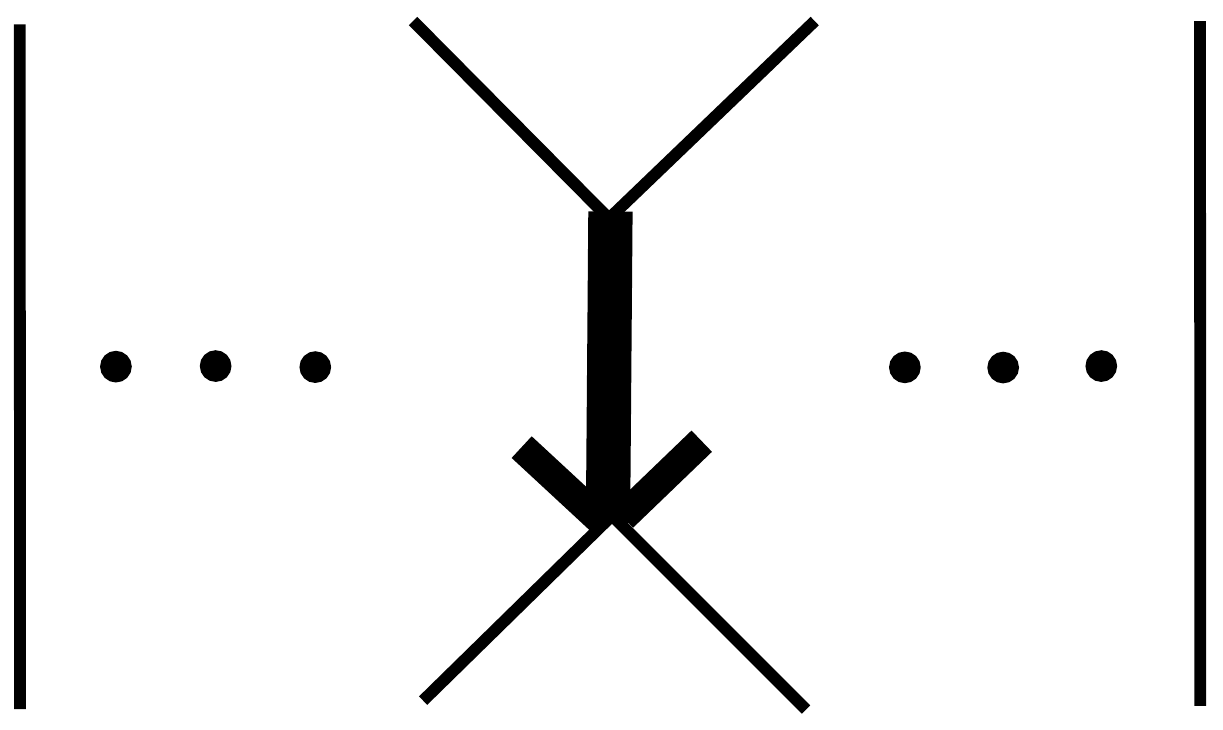}}\hspace{.2cm}=\hspace{.2cm}\raisebox{-.7cm}{\includegraphics[height=.6in]{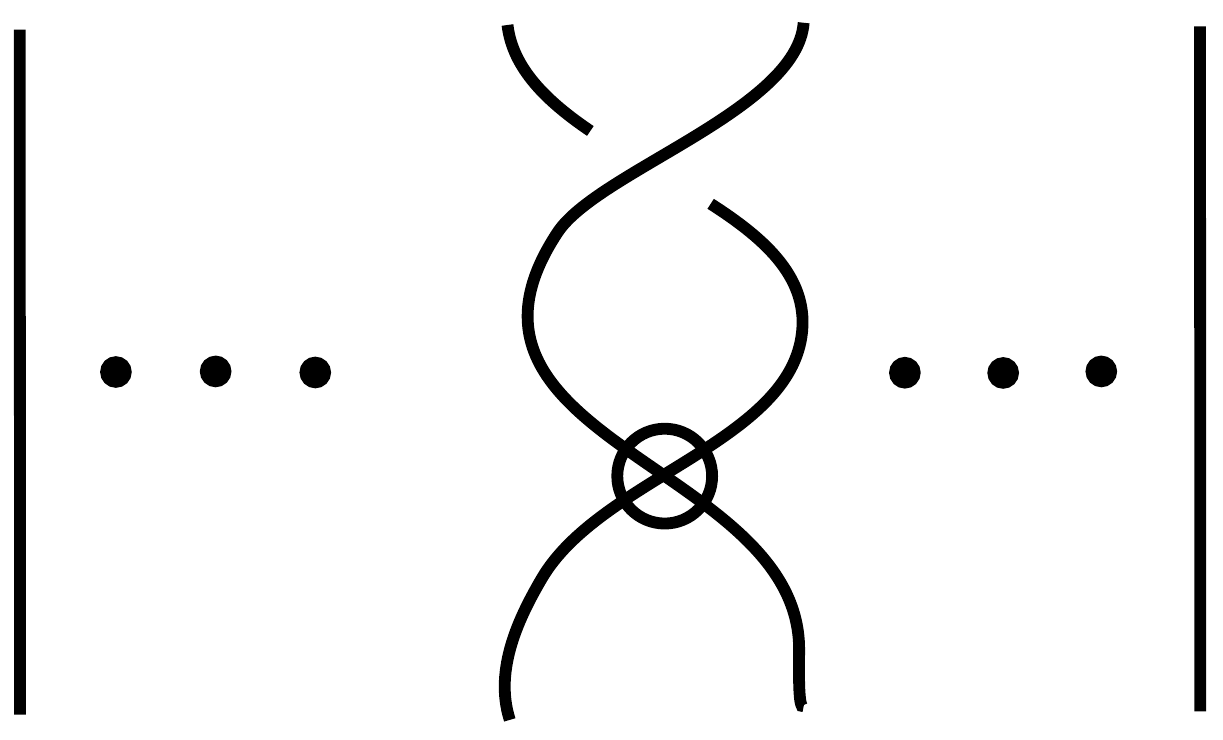}}\]
  \vspace{.2cm}
  \[\mu_i^{-1}\hspace{.2cm}=\hspace{.2cm}\raisebox{-.7cm}{\includegraphics[height=.6in]{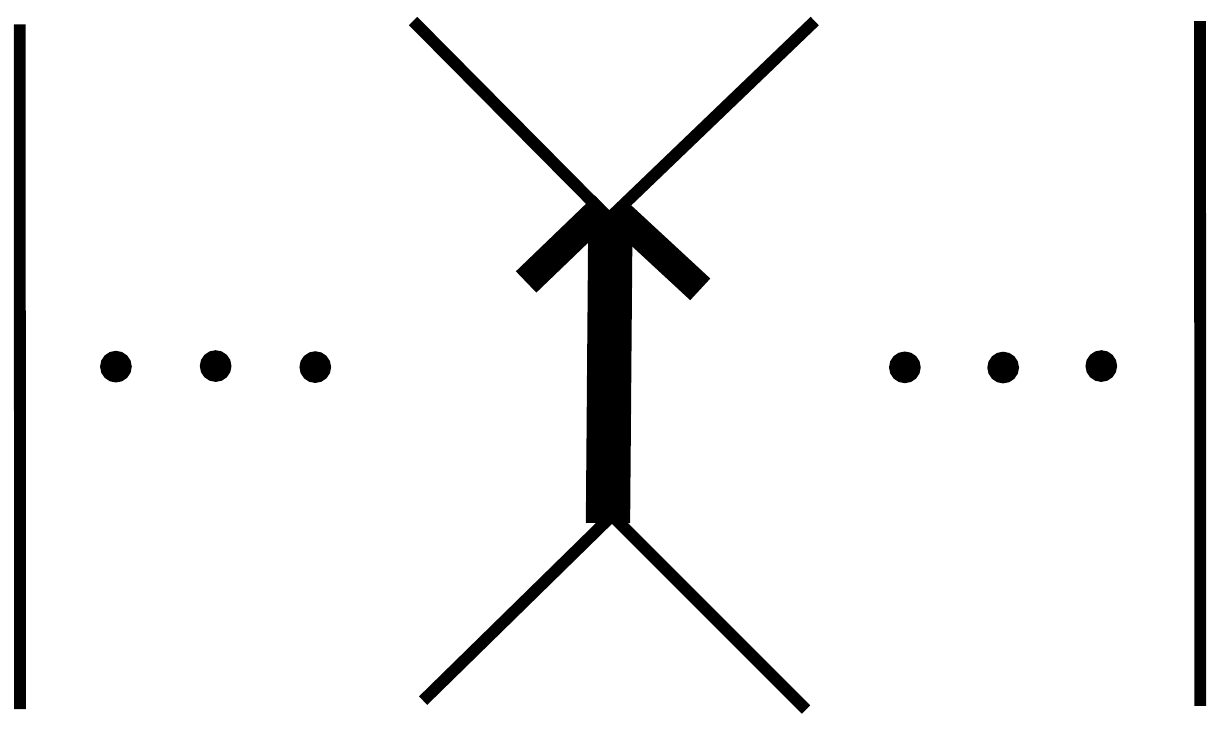}}\hspace{.2cm}=\hspace{.2cm}\raisebox{-.7cm}{\includegraphics[height=.6in]{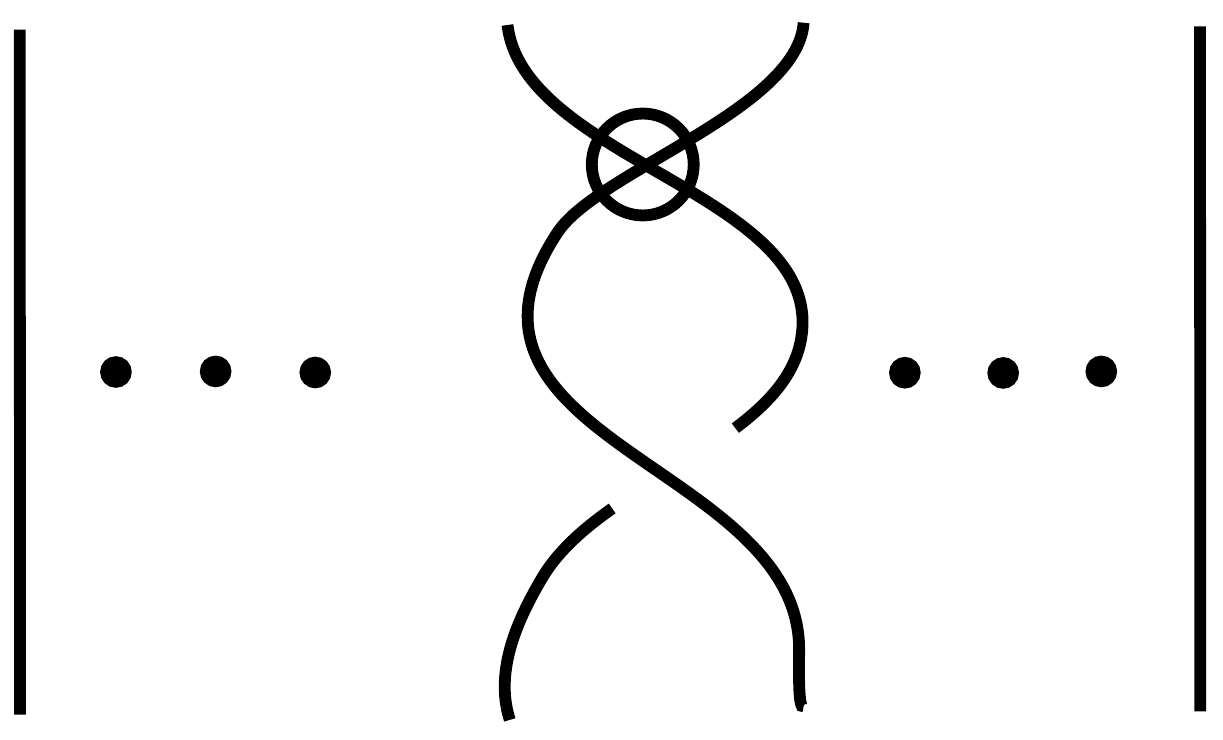}}\hspace{.2cm}\]
  \vspace{.2cm}
  \[\gamma_i\hspace{.2cm}=\hspace{.2cm}\raisebox{-.7cm}{\includegraphics[height=.6in]{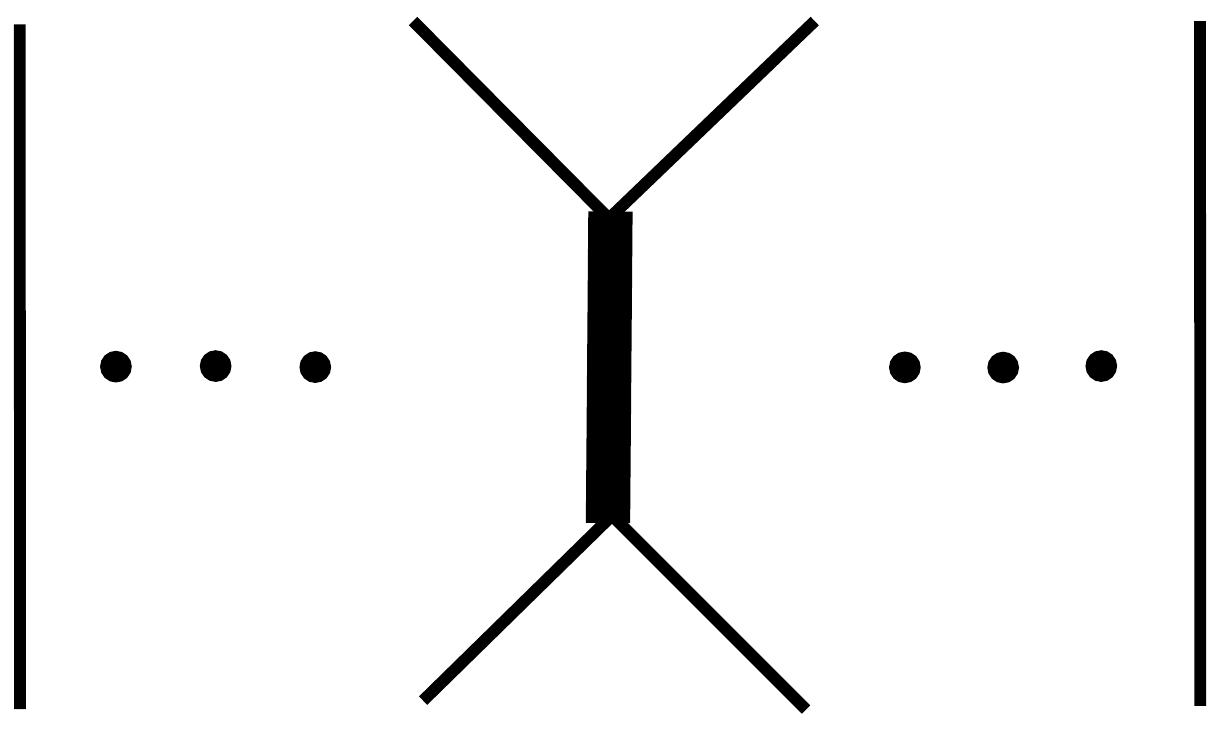}}\hspace{.2cm}=\hspace{.2cm}\raisebox{-.7cm}{\includegraphics[height=.6in]{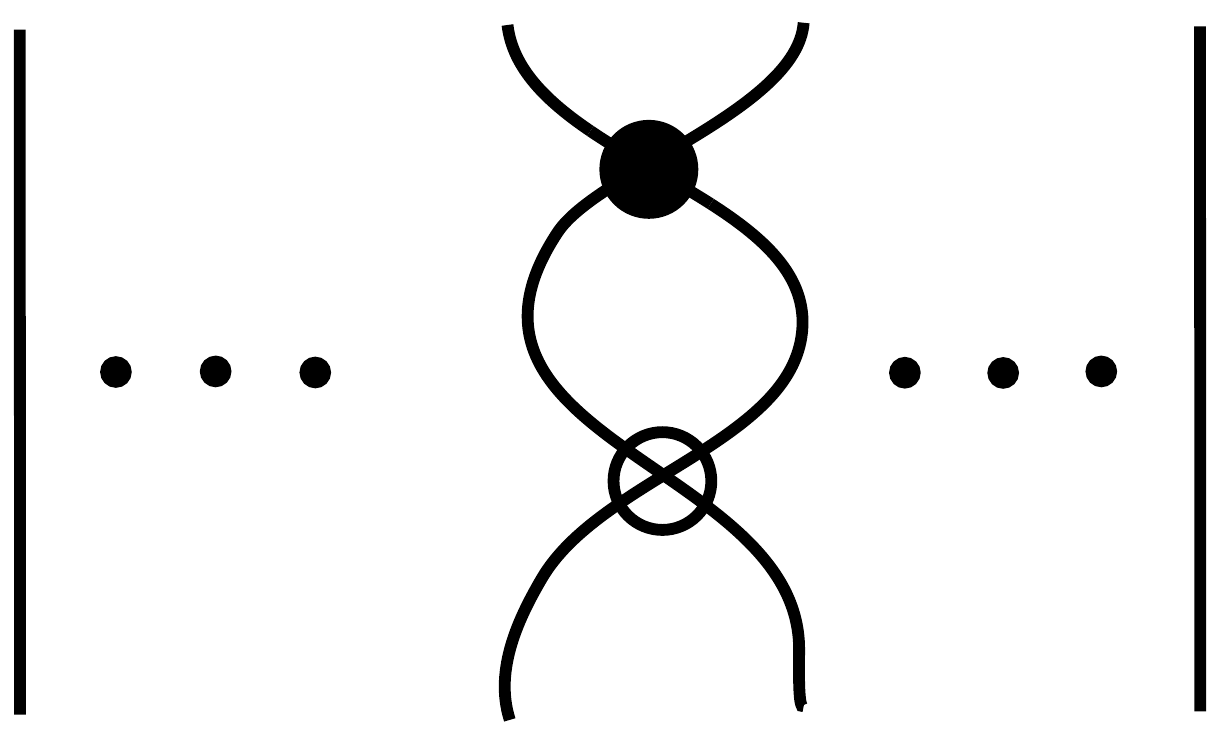}}\]
   \caption{Elementary connecting strings} \label{fig:connecting-strings}
 \end{figure}

  Our motivation for considering the elementary fusing strings stems from the fact that $\mu_i$ is an invertible solution to the algebraic Yang-Baxter equation. Further, there exists a relation between $\mu_i$ and $\gamma_i$ that closely resembles the Yang-Baxter equation. We will return to this later.

It is easy to see that the permutation associated with the elementary connecting strings is the identity permutation of $S_n$. That is, elementary fusing strings are virtual singular pure braids.

Using the defining relations for $\mu_i, \mu_i^{-1}$ and $\gamma_i$ given in Definition~\ref{def:cstrings}, we can describe the elementary virtual singular braids $\sigma_i,\sigma_i^{-1},\text{and } \tau_i$ in terms of the elementary fusing strings:
   \[\sigma_i=\mu_{i}v_i,\,\,\,\,\,\,\,\,\,\,\,\,\,\sigma_i^{-1}=v_i\mu_{i}^{-1},\,\,\,\,\,\,\,\,\,\,\,\,\,\tau_i=\gamma_{i}v_{i}.\]

Thus, the fusing strings $\mu_{i}$, $\mu_i^{-1}$, and $\gamma_{i}$, along with the virtual generators $v_i$, can be used as an alternative set of generators for the $n$-stranded virtual singular braid monoid, $VSB_n$. The following lemmas will provide a set of relations that the elementary fusing strings satisfy. In the proof of each lemma, we underline the portion of the word that is modified in the next line of the proof, and indicate which relation from Definition \ref{def:vsbn} is used to modify it. The first lemma introduces a version of the detour move for elementary fusing strings. 
  
\begin{lemma}\label{lemma:detour}
The following relations hold in $VSB_n$, for all $|i-j|=1$: 
 \[v_i\mu_{j}v_{i}=v_{j}\mu_{i}v_{j}\hspace{2cm}v_i\gamma_{j}v_{i}=v_{j}\gamma_{i}v_{j}\] 
 \[ \raisebox{-22pt}{\includegraphics[height=.7in]{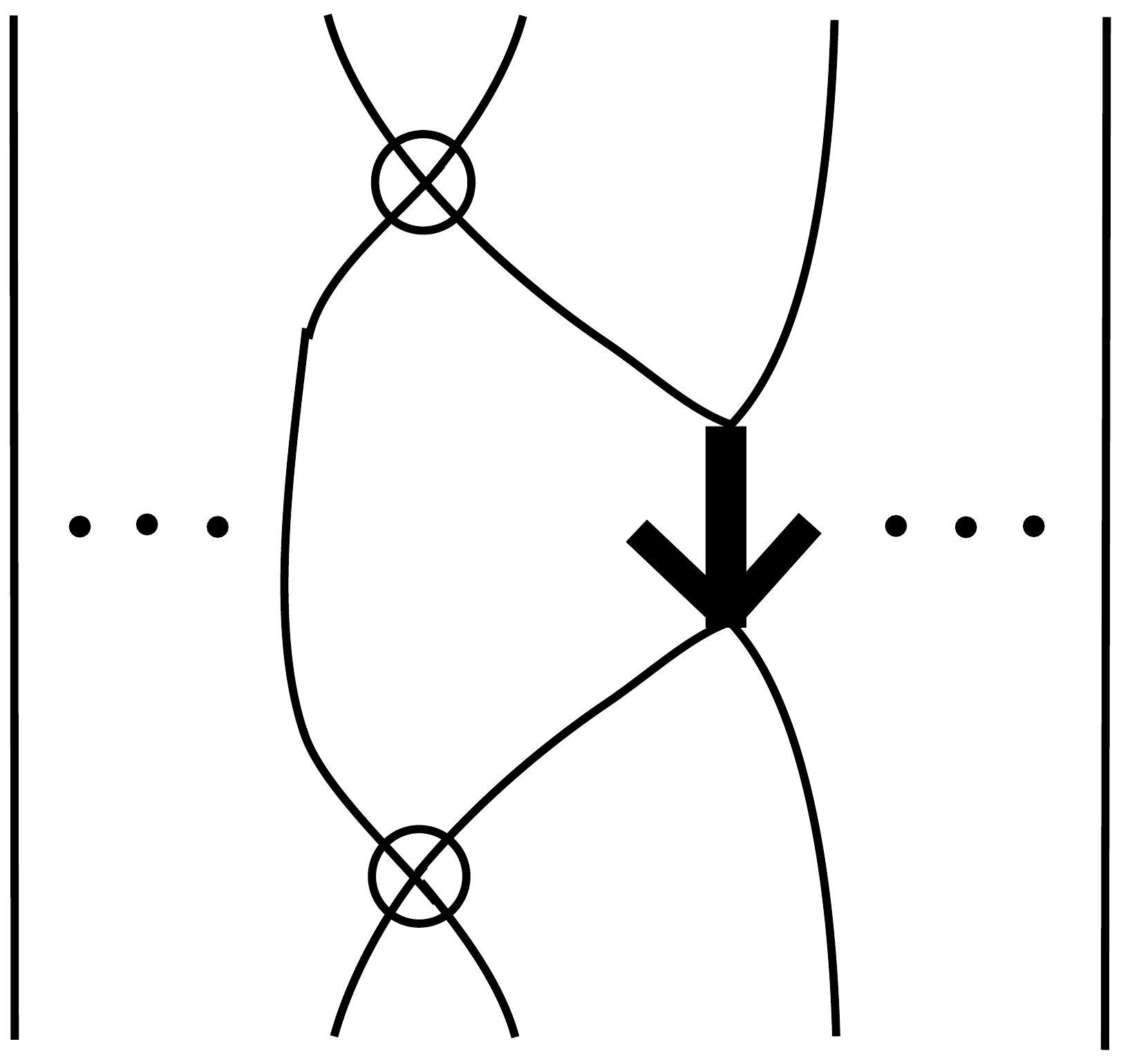}}\hspace{.2cm} \longleftrightarrow\hspace{.2cm}  \raisebox{-22pt}{\includegraphics[height=.7in]{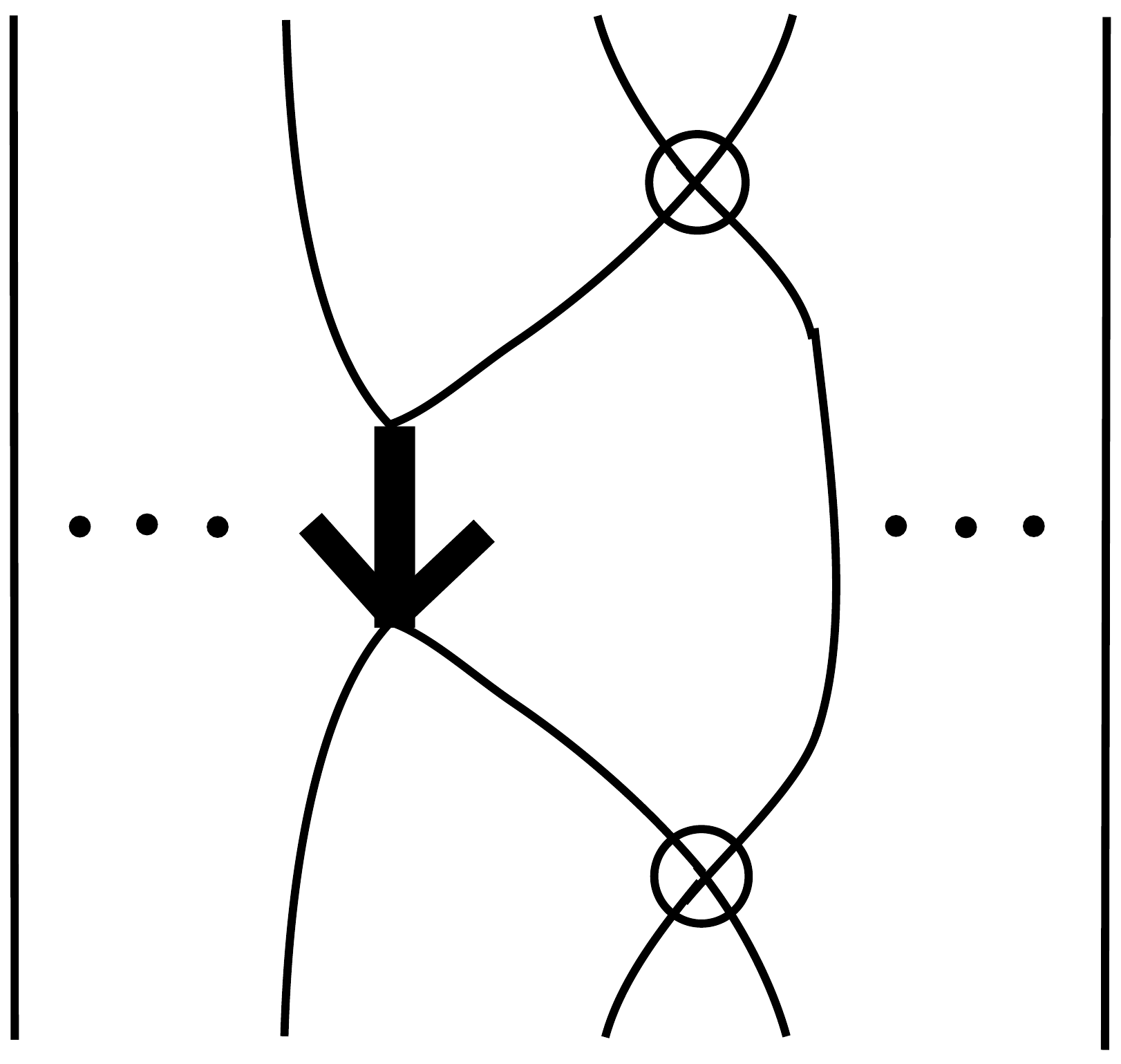}}\hspace{1.8cm}\raisebox{-22pt}{\includegraphics[height=.7in]{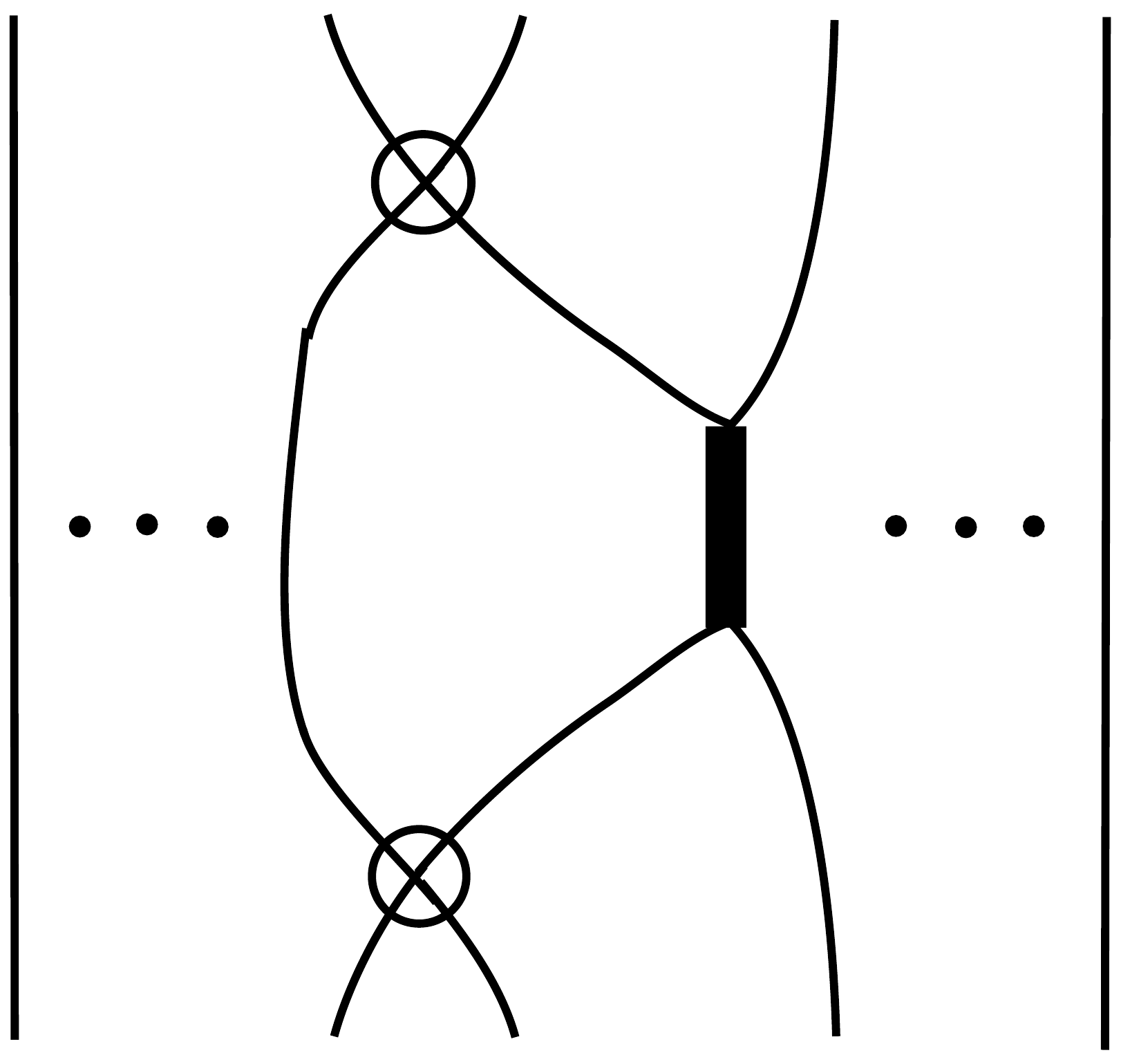}}\hspace{.2cm} \longleftrightarrow\hspace{.2cm}  \raisebox{-22pt}{\includegraphics[height=.7in]{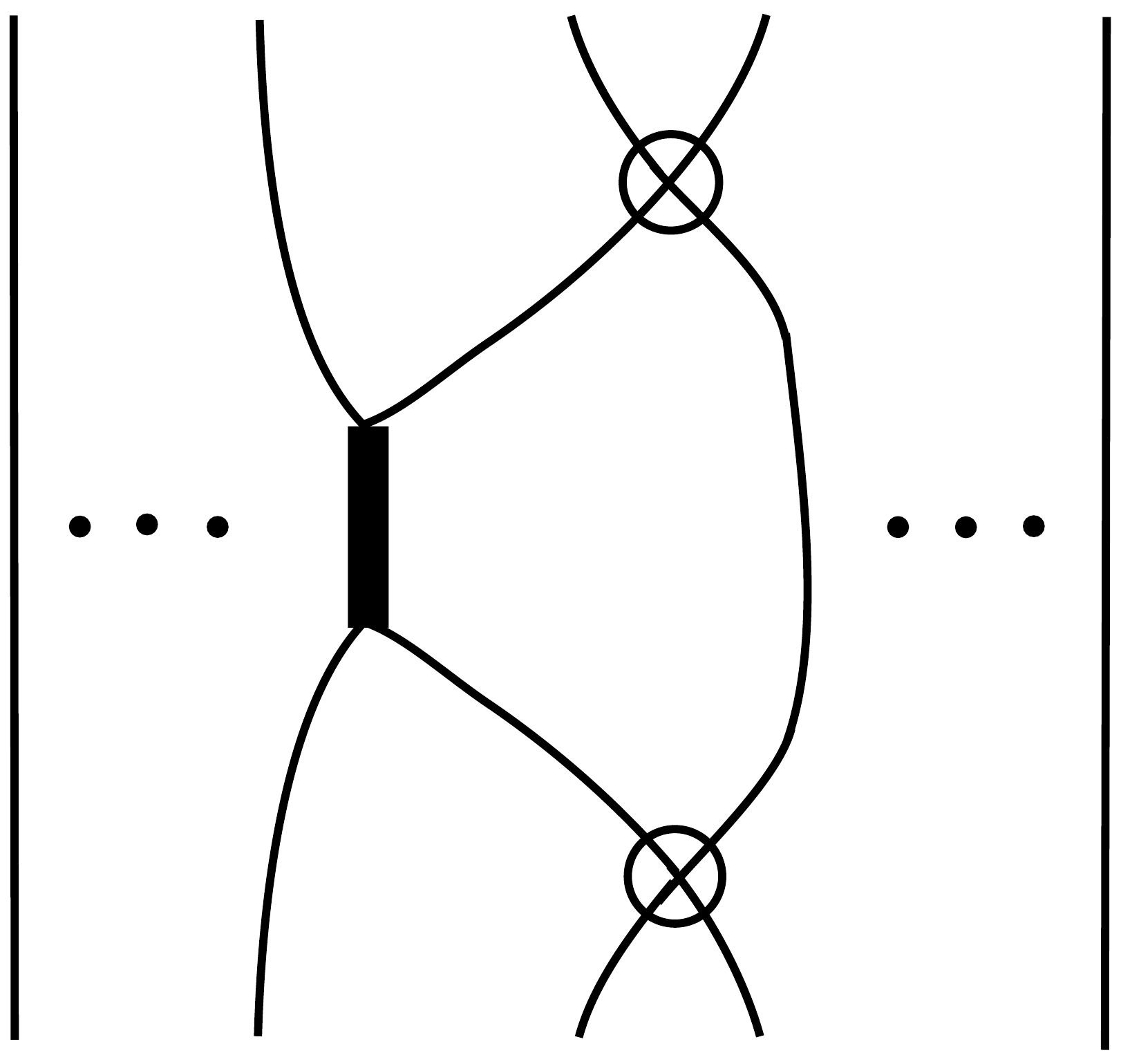}}
\]
\end{lemma}

\begin{proof} The desired relations follow from the detour relations of $VSB_n$.
For the second relation we have:
\begin{align*}
v_i\underline{\gamma_{j}}v_i&\,\,{=}\,\,v_i\tau_{j}v_{j}v_i\underline{1_n}&{V2}\\
&\,\,{=}\,\,v_i\tau_{j}\underline{v_{j}v_iv_{j}}v_{j}&{V3}\\
&\,\,{=}\,\,\underline{v_i\tau_{j}v_{i}}v_{j}v_{i}v_{j}&{VS3}\\
&\,\,{=}\,\,v_{j}\tau_{i}\underline{v_{j}v_{j}}v_{i}v_{j}&{V2}\\
&\,\,{=}\,\,v_{j}\underline{\tau_{i}v_{i}}v_{j}&\\
&\,\,{=}\,\,v_{j}\gamma_iv_{j}.
\end{align*}
The first relation in the statement is proved in the same way, by replacing $\gamma_i$ and $\tau_i$ by $\mu_i$ and $\sigma_i$, respectively. 
\end{proof}
\begin{remark}
For every $|i-j|=1$, the relations 
\[ \mu_{j}v_{i}v_{j}=v_iv_{j}\mu_{i}\hspace{.3in} \text{and}\hspace{.3in} \gamma_{j}v_{i}v_{j}=v_{i}v_{j}\gamma_{i}\]
 are equivalent to the relations of Lemma~\ref{lemma:detour}. We derive them, algebraically, by multiplying the relations in Lemma \ref{lemma:detour} on the left by $v_i$ and on the right by $ v_{j}$, as shown below:
\[{ v_i}(v_i\mu_{j}v_{i})v_{j}={v_i}(v_{j}\mu_{i}v_{j}){v_{j}}\hspace{.3in}\text{and}\hspace{.3in}v_i(v_i\gamma_{j}v_{i})v_{j}={v_i}(v_j\gamma_{i}v_{j})v_{j}.\]

We then apply the identity relations $v_i^2=1_n$ and $v_j^2=1_n$, to obtain the desired relations.
\end{remark}

In the next lemma, we present the Yang-Baxter type relations for fusing strings. Naturally, they follow from the $R3$-type relations of $VSB_n$. 

\begin{lemma}\label{lemma:R3}
The following relations hold in $VSB_n$, for all $|i-j|=1$:
\[ \mu_{j}(v_{j}\mu_{i}v_{j})\mu_i=\mu_{i}(v_{j}\mu_{i}v_{j})\mu_{j}\]
\[\mu_{j}(v_{j}\mu_{i}v_{j})\gamma_i=\gamma_{i}(v_{j}\mu_{i}v_{j})\mu_{j}\]
\end{lemma}

\begin{proof} 
For the first relation we have:
\begin{align*}
\mu_{j}(v_{j}\mu_{i}v_{j})\mu_i&\,\,{=}\,\,\sigma_{j}\underline{v_{j}v_{j}}\sigma_iv_iv_{j}\sigma_iv_i&{V2}\\
&\,\,{=}\,\,\sigma_{j}\sigma_i\underline{v_iv_{j}\sigma_i}v_i&{VR3}\\
&\,\,{=}\,\,\underline{\sigma_{j}\sigma_i\sigma_{j}}\,\,\,\underline{v_iv_{j}v_i}&{R3,V3}\\
&\,\,{=}\,\,\sigma_{i}\sigma_{j}\underline{\sigma_{i}v_{j}v_{i}}v_{j}&{VR3}\\
&\,\,{=}\,\,\sigma_{i}\sigma_{j}v_{j}v_{i}\sigma_{j}v_{j}.
\end{align*}

Next, we insert the identity to obtain the following:
\begin{align*}
\mu_{j}(v_{j}\mu_{i}v_{j})\mu_i&\,\,{=}\,\,\sigma_{i}\sigma_{j}v_{j}v_{i}\underline{1_n}\sigma_{j}v_{j}&{V2}\\
&\,\,{=}\,\,\sigma_{i}\sigma_{j}\underline{v_{j}v_{i}v_{j}}v_{j}\sigma_{j}v_{j}&{V3}\\
&\,\,{=}\,\,\sigma_{i}\underline{\sigma_{j}v_{i}v_{j}}v_{i}v_{j}\sigma_{j}v_{j}&{VR3}\\
&\,\,{=}\,\,\underline{\sigma_{i}v_{i}}v_{j}\underline{\sigma_{i}v_{i}}v_{j}\underline{\sigma_{j}v_{j}}&\\
&\,\,{=}\,\,\mu_i(v_{j}\mu_iv_{j})\mu_{j},
\end{align*}
where the first and last equality hold due to the defining relation for the elementary fusing strings $\mu_i$.

For the second relation we have:
\begin{align*}
\mu_{j}(v_{j}\mu_{i}v_{j})\gamma_i&\,\,{=}\,\,\sigma_{j}\underline{v_{j}v_{j}}\sigma_iv_iv_{j}\tau_iv_i&{V2}\\
&\,\,{=}\,\,\sigma_{j}\sigma_i\underline{v_iv_{j}\tau_i}v_i&{VS3}\\
&\,\,{=}\,\,\underline{\sigma_{j}\sigma_i\tau_{j}}\,\,\,\underline{v_iv_{j}v_i}&{RS3,V3}\\
&\,\,{=}\,\,\tau_{i}\sigma_{j}\underline{\sigma_{i}v_{j}v_{i}}v_{j}&{VR3}\\
&\,\,{=}\,\,\tau_{i}\sigma_{j}v_{j}v_{i}\underline{1_n}\sigma_{j}v_{j}&{V2}\\
&\,\,{=}\,\,\tau_{i}\sigma_{j}\underline{v_{j}v_{i}v_{j}}v_{j}\sigma_{j}v_{j}&\text{V3}\\
&\,\,{=}\,\,\tau_{i}\underline{\sigma_{j}v_{i}v_{j}}v_{i}v_{j}\sigma_{j}v_{j}&{VR3}\\
&\,\,{=}\,\,\underline{\tau_{i}v_{i}}v_{j}\underline{\sigma_{i}v_{i}}v_{j}\underline{\sigma_{j}v_{j}}&\\
&\,\,{=}\,\,\gamma_i(v_{j}\mu_iv_{j})\mu_{j}.
\end{align*}
Again, the first and last equality above hold due to the defining relations for the elementary fusing strings $\mu_i$ and $\gamma_i$.
\end{proof}

\begin{lemma}\label{lemma:fc} 
The following relations hold in $VSB_n$, for $|i-j|>1$:
\[\alpha_i\beta_j=\beta_j\alpha_i,\,\,\,\,\,\text{for } \alpha_i,\beta_i\in \{\mu_i,\gamma_i,v_i\}\]
 \end{lemma}

\begin{proof} These relations are the commuting relations for elementary fusing strings. Naturally, they follow from the commuting relations for $VSB_n$ (see the last set of relations of Definition \ref{def:vsbn}). 
\end{proof}

We now state one more lemma which preserves the singular twist relation of $VSB_n$ among fusing strings.  

\begin{lemma}\label{lemma:twist}
The following relations hold in $VSB_n$, for all $1\leq i \leq n-1$:
\begin{eqnarray*} 
\mu_{i}v_i\gamma_i=\gamma_iv_i\mu_{i}
\end{eqnarray*}
\[ \raisebox{-33pt}{\includegraphics[height=1in]{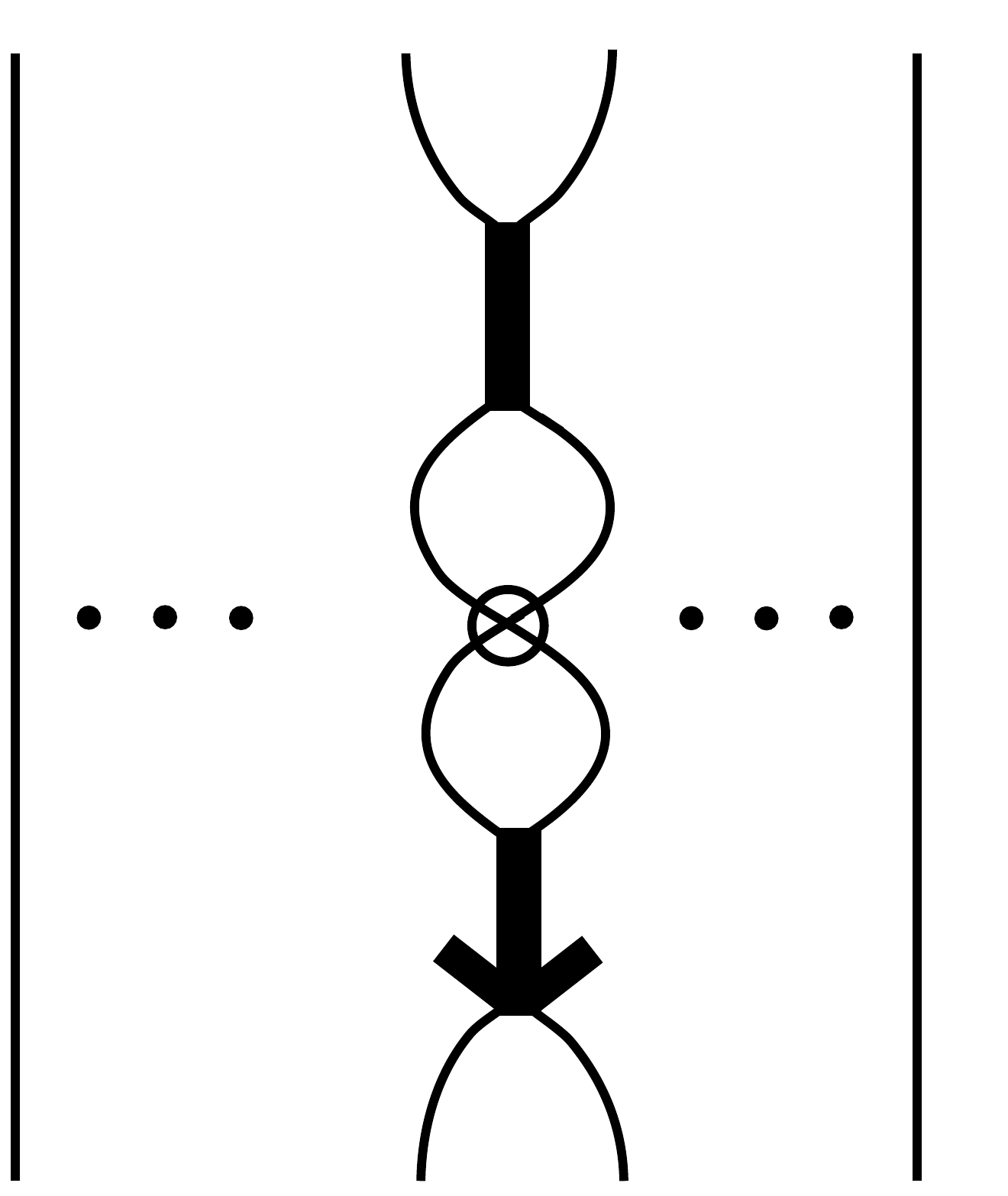}}\hspace{.2cm} \longleftrightarrow\hspace{.3cm}  \raisebox{-33pt}{\includegraphics[height=1in]{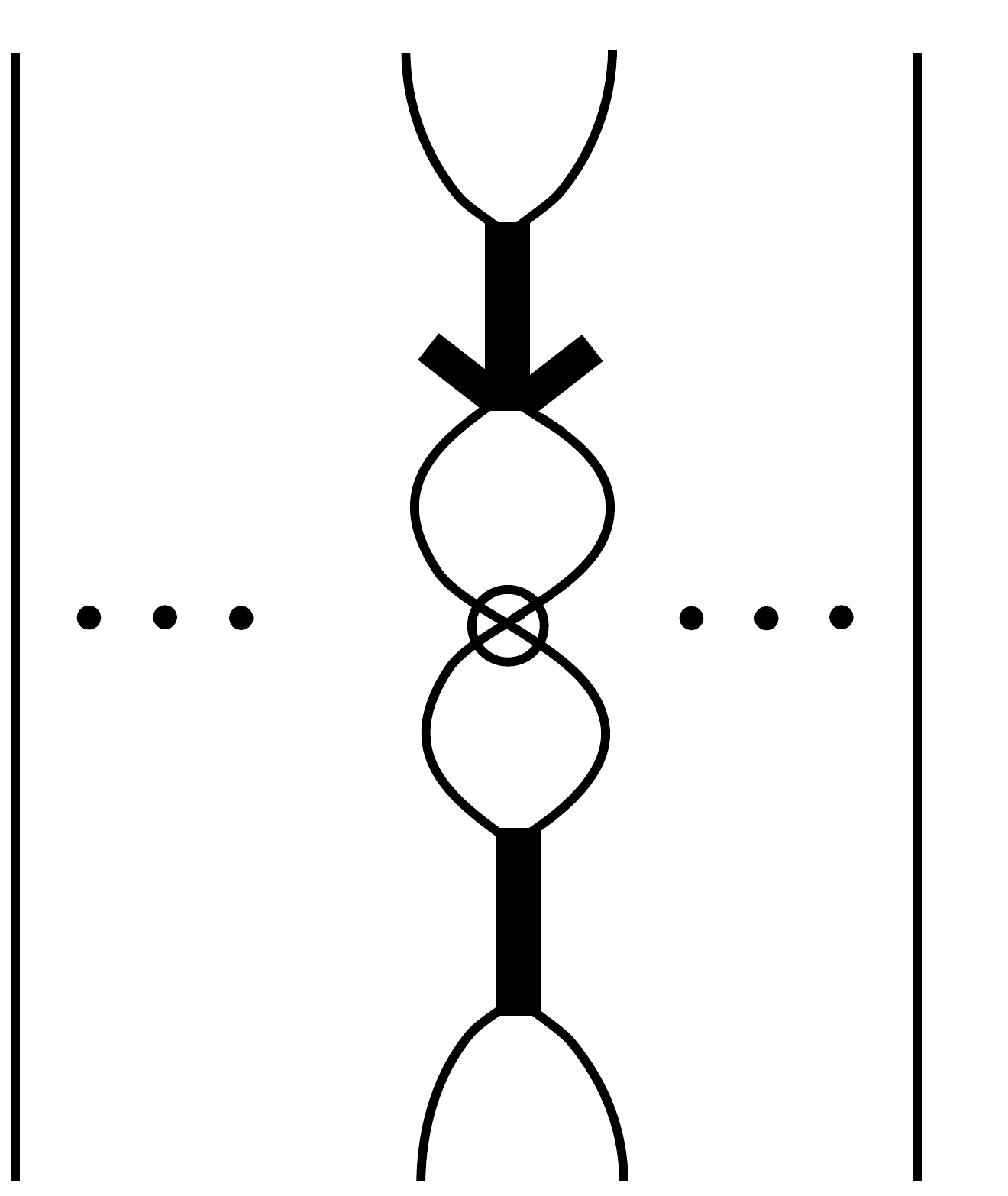}}\]
\end{lemma}

\begin{proof} 
Let $1\leq i \leq n-1$. Using the defining relations for $\mu_i$ and $\gamma_i$, the identity relation $v_i^2=1_n$, and the singular twist relation $\sigma_i\tau_i=\tau_i\sigma_i$, which we know hold in $VSB_n$, we obtain:
\begin{align*}
\mu_iv_i\gamma_i&\,\,{=}\,\,(\sigma_i\underline{v_i)v_i}(\tau_iv_i)&{V2}\\
&\,\,{=}\,\,\underline{\sigma_i\tau_i}v_i&{RS1}\\
&\,\,{=}\,\,\tau_i\underline{1_n}\sigma_iv_i&{V2}\\
&\,\,{=}\,\,(\tau_iv_i)v_i(\sigma_iv_i)&\\
&\,\,{=}\,\,\gamma_iv_i\mu_i.
\end{align*}
Therefore, the relations hold for all $1\leq i \leq n-1$.
\end{proof}

\subsection{A presentation for $VSB_n$ using fusing strings} \label{sec:PresFusing}

We define a monoid with the elementary fusing strings as generators. The relations proved in the preceding lemmas become defining relations for this monoid. In this monoid, $\mu_i$ and $\mu_i^{-1}$  will be inverses of each other. This should be clear from their definitions. In addition, we have the elements $v_i$ of $VSB_n$ as generators for this monoid. Consequently, we must adopt the defining relations of $VSB_n$ involving only the virtual generators. 

\begin{definition} Let $M_n$ be the monoid with the following presentation using generators $\{\mu_i^{\pm1},\gamma_i,v_i\big{|} 1\leq i \leq {n-1}\}$ and relations:
     \begin{align}
     v_i^2=1_n \,\, &\text{and}\,\, \, \mu_i\mu_i^{-1}=1_n=\mu_i^{-1}\mu_i \label{eq:mnid}\\
       v_i v_{j} v_i&= v_{j} v_i v_{j},  \,\, |i-j|=1 \label{eq:mnv3}\\
       v_i \mu_{j} v_i&= v_{j} \mu_i v_{j},  \,\, |i-j|=1\label{eq:mnvr3}\\
 v_i \gamma_{j} v_i&= v_{j} \gamma_i v_{j}, \,\, |i-j|=1\label{eq:mnvs3}\\
 \mu_{j}(v_{j}\mu_{i}v_{j})\mu_{i}&=\mu_{i}(v_{j}\mu_{i}v_{j})\mu_{j}, \,\, |i-j|=1\label{eq:mnr3}\\
       \mu_{j}(v_{j}\mu_{i}v_{j})\gamma_{i}&=\gamma_{i}(v_{j}\mu_{i}v_{j})\mu_{j}, \,\,  |i-j|=1\label{eq:mnrs31}\\
         \mu_{i}v_i\gamma_{i}&=\gamma_{i}v_i\mu_{i}\label{eq:mnr1}\\
    \alpha_i\beta_j&= \beta_j\alpha_i,\,\,\,\,|i-j|>1,\,\,\,\,  \forall \alpha_i,\beta_i\in \{\mu_i,\gamma_i,v_i\}\label{eq:mnfc}
\end{align} 
\label{def:mn}
\end{definition}

\begin{theorem}\label{theorem:iso}
$M_n$ is isomorphic to $VSB_n$.
\end{theorem}   

\begin{proof}
  Define the map $F\co VSB_n\longrightarrow M_n$ by 
 \[F(v_i)=v_i,\,\,\,\, F(\sigma_i) = \mu_{i}v_i, \,\,\,\, F(\sigma_i^{-1})=v_i\mu_i^{-1},\,\,\,\, \text{and}\,\, \,\, F(\tau_i)=\gamma_{i}v_i\]
 and extend it to all virtual singular braids, written in terms of the generators of $VSB_n$, so that $F$ is a homomorphism.
We need to show that $F$ is well defined. That is, we need to show that $F$ preserves the defining relations for $VSB_n$. Since $F$ fixes the virtual generators, the relations among virtual generators in $VSB_n$ are preserved. We will now show that the remaining relations in $VSB_n$ are also preserved. 

 We begin by showing that the identity relation, $R2$, involving real crossings is preserved. We need to show that $F(\sigma_i\sigma_i^{-1})=F(1_n)=F(\sigma_i^{-1}\sigma_i)$. Since $F$ is a homomorphism, we must have $F(1_n)=1_n$. So it suffices to prove that both $F(\sigma_i\sigma_i^{-1}) =  1_n$ and $F(\sigma_i^{-1}\sigma_i)=1_n.$ We have: 
  \[F(\sigma_i\sigma_i^{-1}) = \mu_iv_iv_i\mu_i^{-1}  \stackrel{\eqref{eq:mnid} }{=} 1_n \,\, \, \text{and} \,\, \, F(\sigma_i^{-1}\sigma_i)=v_i\mu_i^{-1}\mu_iv_i  \stackrel{\eqref{eq:mnid} }{=}  1_n.\] 
Thus the relation $R2$ is preserved. Next we show that the singular twist relation $RS1$ holds:
 \[F(\sigma_i\tau_i)=\underline{(\mu_{i}v_i)(\gamma_{i}}v_i) \stackrel{\eqref{eq:mnr1}}{=}(\gamma_iv_i)(\mu_iv_i)=F(\tau_i\sigma_i).\]

As shown below, $F$ preserves the relation $RS3$:

 \begin{eqnarray*}
 F(\sigma_i\sigma_{j}\tau_i)& = & (\mu_{i}\underline{1_n}v_i)(\mu_{j}v_{j})(\gamma_{i}\underline{1_n}v_i)\\
 &\stackrel{\eqref{eq:mnid}}{=}& \mu_{i}v_{j}\underline{v_{j}v_i\mu_{j}}\,\,\underline{v_{j}\gamma_{i}v_{j}}v_{j}v_i\\
 &\stackrel{\eqref{eq:mnvr3}, \eqref{eq:mnvs3}}{=}& \mu_{i}v_{j}\mu_{i}v_{j}\underline{v_{i}v_{i}}\gamma_{j}v_{i}v_{j}v_i\\
  &\stackrel{\eqref{eq:mnid}}{=}&\underline{\mu_{i}(v_{j}\mu_{i}v_{j})\gamma_{j}}\,\,\underline{v_{i}v_{j}v_i}  \\
& \stackrel{\eqref{eq:mnrs31}, \eqref{eq:mnv3}}{=}& \gamma_{j}(v_{j}\mu_{i}\underline{v_{j})\mu_{i}v_{j}}v_{i}v_{j} \\
& \stackrel{\eqref{eq:mnvr3} }{=}& \gamma_{j}v_{j}\mu_{i}v_{i}\mu_{j}v_{i}v_{i}v_{j}.
  \end{eqnarray*}
By applying the identity relation~\eqref{eq:mnid}, we have the desired word as a result of the defining relations for $\mu_i$ and $\gamma_i$:
  \begin{eqnarray*}
    F(\sigma_i\sigma_{j}\tau_i) &= &\gamma_{j}v_{j}\mu_{i}v_{i}\mu_{j}\underline{v_{i}v_{i}}v_{j}\\
    &=&(\gamma_{j}v_{j})(\mu_{i}v_{i})(\mu_{j}v_{j})\\
 &=&F(\tau_{j}\sigma_i\sigma_{j}).
  \end{eqnarray*}
  
  Therefore, $F$ preserves the relation $RS3$. The proof that relation $R3$ is also preserved by $F$ is similar; one replaces $\tau_i$ by $\sigma_ i$ and makes use of relation~\eqref{eq:mnr3} instead of relation~\eqref{eq:mnrs31}.
  
Next we show that the relation $VS3$ is preserved by $F$:
 \begin{eqnarray*}
 F(v_i\tau_{j}v_i)&=&v_i(\gamma_{j}\underline{1_n}v_{j})v_i\\
 &\stackrel{ \eqref{eq:mnid}}{=}& \underline{v_i\gamma_{j}v_i}\,\,\underline{v_iv_{j}v_i} \\
 &\stackrel{\eqref{eq:mnvs3}, \eqref{eq:mnv3}}{=} &v_j\gamma_{i}\underline{v_jv_j}v_{i}v_j\\
  &\stackrel{ \eqref{eq:mnid}}{=}&v_j(\gamma_{i}v_{i})v_j\\
 &=&F(v_{j}\tau_iv_{j}).
  \end{eqnarray*}

The relation $VR3$ can be verified in the same way by replacing $\tau_i$ and $\gamma_i$ with $\sigma_i$ and $\mu_i$, respectively. Instead of using the relation~\eqref{eq:mnvs3}, one uses relation~\eqref{eq:mnvr3}.

It is easy to see that $F$ preserves the commuting relations for $VSB_n$; this follows from relations \eqref{eq:mnfc}.

Therefore, $F$ is a well-defined homomorphism from $VSB_n$ to $M_n$. To show that $F$ is an isomorphism, we define a homomorphism from $M_n$ to $VSB_n$ that is the inverse map of $F$.

We define $G \co M_n\longrightarrow VSB_n$ by 
\[G(v_i)=v_i, \, G(\mu_{i})=\sigma_iv_i,\, G(\mu^{-1}_{i})=v_i\sigma^{-1}_i, \, \text{and }  G(\gamma_{i})=\tau_iv_i,\]
and extend it to all members in $M_n$ so that $G$ is a homomorphism.

We remark that $G$ is defined such that it mimics the defining relations for the elementary fusing strings $\mu_i^{\pm1}$ and $\gamma_i$ (see Definition~\ref{def:cstrings}). By Lemmas~\ref{lemma:detour} and~\ref{lemma:R3}--\ref{lemma:twist}, together with the definition for the monoid $M_n$, it follows that $G$ preserves the relations for $B_n$.

Moreover, it is easy to see that $F\circ G$ and $G\circ F$ are the identity maps on $M_n$ and $VSB_n$, respectively. Therefore, $VSB_n$ and $M_n$ are isomorphic monoids. This completes the proof.
\end{proof}
\begin{remark}
Since $VSB_n$ and $M_n$ are isomorphic monoids, and the generators for $M_n$ are elements in $VSB_n$, our defining presentation for $M_n$ can be used as an alternative presentation for $VSB_n$ using elementary fusing strings.

Similar to Theorem~\ref{theorem:red} for $VSB_n$, there is a reduced presentation for the monoid $M_n$. We arrive at this presentation by using  the detour move to define the generators $\mu_i,\mu_i^{-1}$, and $\gamma_i$ in terms of $\mu_1,\mu_1^{-1}$, and $\gamma_1$, respectively, as follows:
\[\mu_{i+1}^{\pm{1}}: = (v_i\dots v_1)(v_{i+1}\dots v_3v_2)\mu_1^{\pm{1}}(v_2v_3\dots v_{i+1})(v_1v_2\dots v_i)\]
\[\gamma_{i+1}: = (v_i\dots v_1)(v_{i+1}\dots v_3v_2)\gamma_1(v_2v_3\dots v_{i+1})(v_1v_2\dots v_i)\]

This allows us to rewrite all of the relations in Definition~\ref{def:mn} using only the generators $\mu_1^{\pm{1}}$ and $\gamma_1$ along with the virtual generators $v_i$, for all $1\leq i \leq n-1$. Hence, we can describe $M_n$ using fewer generators. Consequently, $VSB_n$ has a reduced presentation using elementary fusing strings $\mu_1^{\pm{1}},\gamma_1$ together with the virtual generators $v_i$, for all $1\leq i \leq n-1$.
\end{remark}

\begin{proposition}\label{prop:redvssn}
$VSB_n$ has the following reduced presentation using $\{v_1,v_2,\dots , v_{n-1}\}$ and the elementary fusing strings $\mu_1^{\pm 1}$ and $\gamma_1$ as generators and subject to the following relations:
\begin{align}
     v_i^2=1_n \,\, &\text{and}\,\, \,  \mu_{1}^{-1}\mu_{1}=1_n = \mu^{-1}_{1}\mu_{1}\\
          v_iv_jv_i&=v_jv_iv_j, \,\, |i-j|=1\\
       (v_1v_2\mu_{1}v_2v_1)(v_2\mu_{1}v_2)\mu_{1}&=\mu_{1}(v_2\mu_{1}v_2)(v_1v_2\mu_{1}v_2v_1)\label{eq:rcsR3} \\
      (v_1v_2\mu_{1}v_2v_1)(v_2\mu_{1}v_2)\gamma_{1}&=\gamma_{1}(v_2\mu_{1}v_2)(v_1v_2\mu_{1}v_2v_1) \\
           \mu_{1}v_1\gamma_{1}&=\gamma_{1}v_1\mu_{1}\\
       v_iv_j&=v_jv_i, \,\, |i-j|>1 \label{eq:fcomv}\\
   \mu_{1} v_i  = v_i \mu_1\,\,\, &\text{and}\,\, \,\gamma_1v_i=v_i\gamma_1, \,\, i \geq 3\\
        \gamma_{1} (v_2 v_1v_3 v_2 \gamma_{1} v_2 v_3 v_1 v_2) &=(v_2 v_1v_3 v_2 \gamma_{1} v_2 v_3 v_1 v_2)\gamma_{1}   \\
        \gamma_{1} (v_2 v_1v_3 v_2 \mu_{1} v_2 v_3 v_1 v_2)& =(v_2 v_1v_3 v_2 \mu_{1} v_2 v_3 v_1 v_2)\gamma_{1}  \\
        \mu_{1} (v_2 v_1v_3 v_2 \mu_{1} v_2 v_3 v_1 v_2) &=(v_2 v_1v_3 v_2 \mu_{1} v_2 v_3 v_1 v_2)\mu_{1}   \label{eq:fcom_mu}
 \end{align} 
\end{proposition}

\begin{proof}
We do not need the generators $\mu_i,\mu_i^{-1},$ and $\gamma_i$ for $2 \leq i \leq n-1$ since we have defined them in terms of the other generators. As for the relations, we see that they are the same as the relations in Definition \ref{def:mn}, where we fix $i=1$ whenever the relations require the generators $\mu_i,\mu_i^{-1},$ or $\gamma_i$. If a relation uses both $i$ and $j$, we assume that the relation occurs at the left most portion of the braid, so we use the smallest possible subscript for $j$. Hence, for the relations that require $|i-j|=1$, we fix $j=2$ but describe the corresponding generator in terms of a generator with a subscript of 1. For example, in the identity \eqref{eq:rcsR3}, we use $v_1v_2\mu_{1}v_2v_1$ in place of $\mu_2$. Likewise, for the relations that require $|i-j|>1$, we fix $j=3$ but describe the corresponding generator in terms of a generator with a subscript of 1. Any relation that occurs elsewhere in the braid follows from these relations by the detour move. We do not include the identities \eqref{eq:mnvr3} and \eqref{eq:mnvs3} in this reduced presentation, since they were used to obtain the defining relations for $\mu_{i+1}^{\pm{1}}$ and $\gamma_{i+1}$. However, the remaining relations of Definition \ref{def:mn} are included in this reduced presentation, where the subscripts have been changed as described above. Note that the commuting relations~\eqref{eq:mnfc} are now represented by the relations~\eqref{eq:fcomv}--\eqref{eq:fcom_mu}.
Therefore, the statement holds.
\end{proof}


\subsection{A connection with Yang-Baxter equation} \label{sec:YBEincentive}

It is interesting and important to remark that the fusing strings are intimately related to the algebraic Yang-Baxter equation, as we now explain (see also~\cite{KL2}).

Let $V$ be a vector space over a field $\mathbb{K}$ and $R \co V \otimes V \to V \otimes V$ a linear map.
Let $\mathcal{T} \co V \otimes V \to V \otimes V$ be the \textit{twist map} given by:
 \[ \mathcal{T}(v \otimes w) = w \otimes v, \hspace{1cm} \forall v, w \in V \] 

Define the operators $R_{12}, R_{23}$ and $R_{13}$ as below,
  \begin{eqnarray*}
R_{12}: &=& R \otimes \id_V  \\ 
 R_{23}: &=& \id_V \otimes R  \\ 
 R_{13}: &=& (\id_V \otimes \mathcal{T})(R \otimes \id_V)(\id_V \otimes \mathcal{T}) 
 \end{eqnarray*}
and recall that $R$ is said to satisfy the \textit{Yang-Baxter equation} (YBE) if
 \begin{eqnarray} \label{eq:YBE0}
  R_{12}R_{13}R_{23} = R_{23} R_{13}R_{12} 
 \end{eqnarray}
 
 Using the replacements
 \[ R_{12} \rightsquigarrow \mu_1, \, R_{23} \rightsquigarrow \mu_2, \, R_{13} \rightsquigarrow v_2\mu_1v_2, \]
the YBE~\eqref{eq:YBE0} is equivalent to the equation
\begin{eqnarray} \label{eq:YBE1}
 \mu_1 (v_2\mu_1v_2)\mu_2 = \mu_2 (v_2\mu_1v_2)\mu_1, 
 \end{eqnarray}
which is the first relation in Lemma~\ref{lemma:R3} (for the case $i = 1$ and $j=2$).    
This is not unexpected, since it is known that the operator $R$ satisfies the YBE if and only if $B : = \mathcal{T} R$ satisfies the \textit{braid equation}:
\[ B_{12}B_{23}B_{12} = B_{23}B_{12}B_{23}. \]
To this end, it is clear that by making use of the following obvious associations:
\[ B_{12} \rightsquigarrow \sigma_1, \, B_{23} \rightsquigarrow \sigma_2,  \]
the braid equation is equivalent to the well-known relation 
\[\sigma_1 \sigma_2 \sigma_1 =\sigma_2 \sigma_1 \sigma_2,\]
which holds in $VSB_n$ (and in the classical braid group, for that matter).

Relation~\eqref{eq:YBE1} can be thought of as the algebraic Yang-Baxter equation, and this is our main incentive for working with presentations for $VSB_n$ using fusing strings as generators. We will return to Yang-Baxter equation in Section~\ref{sec:repres}, where we  construct a representation of $VSB_n$ into linear operators over $V^{\otimes n}$.

\section{A presentation for the virtual singular pure braid monoid} \label{sec:pure braids}

In this section we find a presentation for the virtual singular pure braid monoid, $VSP_n$, defined in Section~\ref{sec:introVSBn}. To accomplish  this, we introduce the \textit{generalized fusing strings} $\mu_{ij}, \mu_{ji}, \gamma_{ij}$ and $\gamma_{ji}$, where $1 \leq i < j \leq n$. These are elements in $VSP_n$ defined as follows:
\[\mu_{ij}: = (v_{j-1}v_{j-2} \dots v_{i+1}) \mu_i (v_{i+1} \dots v_{j-2} v_{j-1}) \]
\[\gamma_{ij}: = (v_{j-1}v_{j-2} \dots v_{i+1}) \gamma_i (v_{i+1} \dots v_{j-2} v_{j-1}) \]
\[\mu_{ji}: = (v_{j-1}v_{j-2} \dots v_{i+1}) v_i \mu_i v_i (v_{i+1} \dots v_{j-2} v_{j-1}) \]
\[\gamma_{ji}: = (v_{j-1}v_{j-2} \dots v_{i+1}) v_i \gamma_i v_i (v_{i+1} \dots v_{j-2} v_{j-1}) \]

We remark that $\mu_{i, i+1} = \mu_i$ and $\gamma_{i, i+1} = \gamma_i$. Moreover, $\mu_{i+1, i} = v_i \mu_i v_i = v_i \sigma_i$ and $\gamma_{i+1, i} = v_i \gamma_i v_i = v_i \tau_i$.

The group $S_n$ acts on $VSB_n$ by conjugation. The next statement studies this action on the set of generalized fusing strings. Relations 1 below are an immediate consequence of the commuting relations in $VSB_n$ (last set of relations in Definition~\ref{def:vsbn}), while relations 2, 3 and 4 follow from the defining relations for the generalized fusing strings.

\begin{lemma} \label{lemma:action}
The following relations hold in $VSB_n$:
\begin{enumerate}
\item $v_i \, \mu_{kl} \, v_i = \mu_{kl}, \,\,\,\, v_i \, \gamma_{kl} \, v_i = \gamma_{kl}$ for all $|k - i|>1, \, |l - i|>1$
\item $v_{i-1} \, \mu_{i, i+1} \, v_{i-1} = \mu_{i-1, i+1}, \,\,\,\, v_{i-1}  \, \gamma_{i, i+1} \, v_{i-1} = \gamma_{i-1, i+1}$\\
         $v_{i-1} \, \mu_{i+1, i} \,v_{i-1} = \mu_{i+1, i-1}, \,\,\,\, v_{i-1} \, \gamma_{i+1, i} \, v_{i-1} = \gamma_{i+1, i-1}$
\item $v_i \, \mu_{i, i+1} \, v_i = \mu_{i+1, i}, \,\,\, v_i \, \gamma_{i, i+1}\,  v_i = \gamma_{i+1, i}$\\
         $v_i \, \mu_{i+1, i} \, v_i = \mu_{i, i+1}, \,\,\, v_i \, \gamma_{i+1, i} \, v_i = \gamma_{i, i+1}$
\item $v_{i+1} \, \mu_{i, i+1} \, v_{i+1} = \mu_{i, i+2}, \,\,\,\, v_{i+1}\, \gamma_{i, i+1} \, v_{i+1} = \gamma_{i, i+2}$\\
         $v_{i+1} \, \mu_{i+1, i} \, v_{i+1} = \mu_{i+2, i}, \,\,\,\, v_{i+1}\, \gamma_{i+1, i} \, v_{i+1} = \gamma_{i+2, i}$
\end{enumerate}
\end{lemma}

\begin{corollary} \label{cor:action}
The group $S_n$ acts by conjugation on the set $\{ \mu_{kl}, \gamma_{kl} \, | \, 1 \leq k \neq l \leq n \}$, which results in permuting the indices of the generalized fusing strings. That is,
\[\alpha \,  \mu_{kl} \, \alpha^{-1} = \mu_{\alpha(k) \alpha(l)}\,\,\, \text{and} \,\,\, \alpha \,  \gamma_{kl} \, \alpha^{-1} = \gamma_{\alpha(k) \alpha(l)}, \]
where $\alpha \in S_n$. This action is transitive.
\end{corollary}

We are ready now to find a presentation for $VSP_n$ using the Reidemeister-Schreier method (see \cite[Chapter 2.2]{MKS}). By Corollary~\ref{cor:action}, the submonoid of $VSP_n$ generated by the generalized fusing strings is normal in $VSB_n$. We show that this submonoid coincides with $VSP_n$ and find the relations that its generators satisfy. 

\begin{theorem} \label{prespure}
The pure virtual singular braid monoid, $VSP_n$ is generated by the generalized fusing strings $\{ \mu_{kl}, \mu_{kl}^{-1}, \gamma_{kl} \, | \, 1 \leq k \neq l \leq n \}$, subject to the following relations:
\begin{eqnarray}
\mu_{kl} \mu_{kl}^{-1} &=&  \mu_{kl}^{-1} \mu_{kl} = 1_n   \label{eq:pureinv}\\ 
\mu_{ij}\mu_{ik}\mu_{jk} &=& \mu_{jk} \mu_{ik} \mu_{ij}  \label{eq:pureYB} \\
\mu_{ij} \mu_{ik} \gamma_{jk} &=& \gamma_{jk} \mu_{ik} \mu_{ij}  \label{eq:pureYB-mixt}\\
\gamma_{ij} \mu_{ik} \mu_{jk} &=& \mu_{jk} \mu_{ik} \gamma_{ij}  \label{eq:pureYB-mixt-two}\\
\mu_{kl} \gamma_{lk} &=& \gamma_{kl} \mu_{lk} \label{eq:pureTwist} \\
\mu_{ij} \mu_{kl} = \mu_{kl} \mu_{ij} \, ,\,\,\,  \,\,\,\, \gamma_{ij} \gamma_{kl} &=& \gamma_{kl} \gamma_{ij}  \, ,\,\,\, \,\,\,\, \mu_{ij} \gamma_{kl} = \gamma_{kl} \mu_{ij} \,, \label{eq:purecommuting}
\end{eqnarray}
where distinct letters stand for distinct indices.
\end{theorem}

\begin{proof}
We will borrow some notations from~\cite{Ba}.

Let $m_{kl} = v_{k-1}v_{k-2} \dots v_l$ if $l < k$ and $m_{kl} = 1$ in all other cases, and let
\[ \Lambda_n : = \left\{ \prod_{k=2}^n m_k, j_k  \, | \, 1 \leq j_k \leq k  \right \}. \]
Then $\Lambda_n$ is a Schreier system of right coset representatives of $VSP_n$ in $VSB_n$. A Schreier right coset representative is such that any initial segment of a representative is again a representative. 

Define the map $\bar{\empty} : VSB_n \to \Lambda_n$ which takes a braid element $\omega \in VSB_n$ to its representative $\overline{\omega} \in \Lambda_n$. It is easy to see that $\omega \, \overline{\omega}\,^{-1} \in VSP_n$, for all $\omega \in VSB_n$. According to~\cite[Theorem 2.7]{MKS}, the monoid $VSP_n$ is generated by
\[ s_{\lambda, a} = \lambda a \, (\overline{\lambda a})^{-1},\]
for all $\lambda \in \Lambda_n$ and for all generators $a$ of $VSB_n$ (that is, elementary virtual singular braids $\sigma_i, \sigma_i^{-1}, \tau_i$ and $v_i$). Note that for the purpose of this proof, we work here with the standard presentation for $VSB_n$ given in Definition~\ref{def:vsbn}. We also remark that although not all elements in $VSB_n$ are invertible, $\overline{\omega}$ is invertible for any $\omega \in VSB_n$, and therefore $s_{\lambda, a}$ is well defined for all generators, including $\tau_i$.

By definition of $\Lambda_n$, $\overline{\lambda v_i} = \lambda v_i$ for all $\lambda \in \Lambda_n$ and $1 \leq i \leq n-1$. In addition, $\overline{\lambda \sigma_i} = \overline{\lambda \sigma_i^{-1}} = \overline{\lambda \tau_i} = \lambda v_i $ for all $1 \leq i \leq n-1$.
Then, for all $\lambda \in \Lambda_n$, we have:
 \begin{eqnarray*}
 s_{\lambda, v_i} &=& \lambda v_i (\overline{\lambda v_i})^{-1} = \lambda v_i (\lambda v_i)^{-1} = \lambda v_i v_i \lambda^{-1} = 1_n\\
  s_{\lambda, \sigma_i} &=& \lambda \sigma_i (\overline{\lambda \sigma_i})^{-1} = \lambda \sigma_i (\lambda v_i)^{-1} = \lambda (\sigma_i v_i)     \lambda^{-1} = \lambda \mu_{i, i+1} \lambda^{-1} \\
  s_{\lambda, \tau_i} &=& \lambda \tau_i (\overline{\lambda \tau_i})^{-1} = \lambda \tau_i (\lambda v_i)^{-1} = \lambda (\tau_i v_i) \lambda^{-1} = \lambda \gamma_{i, i+1} \lambda^{-1}\\
  s_{\lambda, \sigma_i^{-1}} &=& \lambda \sigma_i^{-1} (\overline{\lambda \sigma_i^{-1}})^{-1} = \lambda \sigma_i^{-1} (\lambda v_i)^{-1} = \lambda (\sigma_i^{-1} v_i) \lambda^{-1} = \lambda \mu_{i+1, i}^{-1} \lambda^{-1}.
  \end{eqnarray*}
Thus for $\lambda = 1_n$, we get that $s_{1_n, \sigma_i}, s_{1_n, \gamma_i}$ and $s_{1_n, \sigma_i^{-1}}$ are the fusing strings $\mu_{i, i+1}, \gamma_{i, i+1}$ and $\mu_{i+1, i}^{-1}$, respectively. By Corollary~\ref{cor:action}, we have that every $s_{\lambda, \sigma_i}$ is equal to some $\mu_{kl}$ and, conversely, each $\mu_{kl}$ is equal to some $s_{\lambda, \sigma_i}$. Similarly, each $s_{\lambda, \tau_i}$ and $s_{\lambda, \sigma_i^{-1}}$ are equal to some $\gamma_{kl}$ and, respectively, $\mu_{kl}^{-1}$. Conversely, each $\gamma_{kl}$ and $\mu_{kl}^{-1}$ are equal to some $s_{\lambda, \tau_i}$ and $s_{\lambda, \sigma_i^{-1}}$, respectively. Therefore, $VSP_n$ is generated by $\{ \mu_{kl}, \mu_{kl}^{-1}, \gamma_{kl} \, | \, 1 \leq k \neq l \leq n \}$.
 
Next we need to find the defining relations for $VSP_n$. Continuing with the Reidemeister-Schreier method, we define a rewriting process $\mathcal{R}$ which converts a word in $VSP_n$ written in terms of the generators of $VSB_n$ into a word written in terms of the generators of $VSP_n$. 

If $ \omega = a_1 a_2 \cdots a_t, \,\,\, \text{where}  \,\, a_j \in \{\sigma_i, \sigma_i^{-1}, \tau_i, v_i  \, | \, 1 \leq i \leq n-1 \},$
 then the rewritten word $\mathcal{R}(\omega)$ is given by
 \[ \mathcal{R}(\omega) = s_{k_1, a_1}^{\epsilon_1} s_{k_2, a_2}^{\epsilon_2} \cdots s_{k_t, a_t}^{\epsilon_t}  \]
 where, for $1 \leq j \leq t$, 
 \begin{itemize}
  \item $\epsilon_j = -1$ and $k_j =  \overline{a_1 a_2 \cdots a_{j-1} a_j}$ \, if $a_j = \sigma_i^{-1}$, for some $1 \leq i \leq n-1$, and
 \item $\epsilon_j = 1$ and $k_j =  \overline{a_1 a_2 \cdots a_{j-1}}$ \, if $a_j$ is any of the other type of generators.
 \end{itemize}

 By~\cite[Theorem 2.8]{MKS}, the defining relations for the monoid $VSP_n$ are given by $\mathcal{R}(\lambda r_u \lambda^{-1})$, for all $\lambda \in \Lambda_n$ and all defining relations $r_u$ of $VSB_n$. (We remark that Theorem 2.8 in~\cite{MKS} is stated for groups; however, its proof works for monoids as well, as long as $\overline{\omega}$ is invertible and $\mathcal{R}(\omega)$ is well defined; this is certainly true in our case.) We find first all of the relations $\mathcal{R}(r_u)$; that is, relations corresponding to $\lambda = 1_n$. Then, the relations for nontrivial $\lambda \in \Lambda_n$ are obtained by conjugating---by representatives in $\Lambda_n$---the previously obtained relations $\mathcal{R}(r_u)$.
 
 Consider the relation $r_1: \sigma_i \sigma_i^{-1} = \sigma_i^{-1} \sigma_i = 1_n$.  Clearly, $\mathcal{R}(1_n) = 1_n$. Then,

 \begin{eqnarray*}
 \mathcal{R}(\sigma_i \sigma_i^{-1}) &=& s_{1_n, \sigma_i} \cdot s_{\overline{\sigma_i \sigma_i^{-1}}, \sigma_i}^{-1} = s_{1_n, \sigma_i} \cdot s_{1_n, \sigma_i}^{-1} = \mu_{i, i+1} \, \mu_{i, i+1}^{-1}\\
  \mathcal{R}(\sigma_i^{-1} \sigma_i) &=& s_{\overline{\sigma_i^{-1}}, \sigma_i}^{-1} \cdot s_{\overline{\sigma_i^{-1}}, \sigma_i} =s_{v_i, \sigma_i}^{-1} \cdot s_{v_i, \sigma_i} \\
    &=& (v_i \mu_{i, i+1} v_i)^{-1} (v_i \mu_{i, i+1} v_i) = \mu_{i+1, i}^{-1} \mu_{i+1, i} 
  \end{eqnarray*} 
  where the last equality above holds by Lemma~\ref{lemma:action}. Thus we have that 
  \[ \mathcal{R}(r_1): \mu_{i, i+1} \, \mu_{i, i+1}^{-1} = \mu_{i+1, i}^{-1} \mu_{i+1, i} = 1_n.\]
Conjugating by $v_i$, we obtain 
\[  \mathcal{R}(v_ir_1v_i): \mu_{i+1, i} \, \mu_{i+1, i}^{-1} = \mu_{i, i+1}^{-1} \mu_{i, i+1} = 1_n. \]
Conjugating by all representatives in $\Lambda_n$ we obtain the relations~\eqref{eq:pureinv}.

The rewriting process applied to relations $r_2: v_i^2 = 1_n$ and $r_4: v_iv_{i+1} v_i = v_{i+1} v_i v_{i+1}$ produce the trivial relation $1_n = 1_n$, since $s_{\lambda, v_i} = 1_n$ for any $\lambda \in \Lambda_n$.

We consider next the relation $r_3: \sigma_i \sigma_{i+1} \sigma_i = \sigma_{i+1} \sigma_i \sigma_{i+1}$. Applying the rewriting process and Corollary~\ref{cor:action}, we have:
\begin{eqnarray*}
\mathcal{R} (\sigma_i \sigma_{i+1} \sigma_i ) &=& s_{1_n, \sigma_i} \, s_{\overline{\sigma_i}, \sigma_{i+1}} \, s_{\overline{\sigma_i \sigma_{i+1} }, \sigma_i} = s_{1_n, \sigma_i} \, s_{v_i, \sigma_{i+1}} \, s_{v_i v_{i+1, \sigma_i}} \\
 &=& (\mu_{i,i+1}) (v_i \mu_{i+1, i+2} v_i) (v_i v_{i+1}\mu_{i,i+1} v_{i+1}v_i) \\
 &=& \mu_{i, i+1} \, \mu_{i, i+2} \, \mu_{i+1, i+2},
\end{eqnarray*}
and 
\begin{eqnarray*}
\mathcal{R} (\sigma_{i+1} \sigma_i \sigma_{i+1} ) &=& s_{1_n, \sigma_{i+1}} \, s_{\overline{\sigma_{i+1}}, \sigma_i} \, s_{\overline{\sigma_{i+1} \sigma_i}, \sigma_{i+1} }
= s_{1_n, \sigma_{i+1}} \, s_{v_{i+1}, \sigma_i} s_{v_{i+1} v_i, \sigma_{i+1}} \\
&=& (\mu_{i+1, i+2})(v_{i+1} \mu_{i, i+i} v_{i+1}) (v_{i+1} v_i \mu_{i+1, i+2}  v_i v_{i+1}) \\
&=& \mu_{i+1, i+2} \, \mu_{i, i+2} \, \mu_{i, i+1}.
\end{eqnarray*}
Therefore, $\mathcal{R}(r_3)$ is represented by $\mu_{i, i+1} \, \mu_{i, i+2} \, \mu_{i+1, i+2} = \mu_{i+1, i+2} \, \mu_{i, i+2} \, \mu_{i, i+1}.$ Conjugating this identity by all representatives in $\Lambda_n$, will produce the relations~\eqref{eq:pureYB}. 

Relations $r_5: v_i \sigma_{i+1} v_i = v_{i+1} \sigma_i v_{i+1}$ and $r_6: v_i \tau_{i+1} v_i = v_{i+1} \tau_i v_{i+1}$ induce the trivial relation in $VSP_n$. We show this for the latter relation.
\begin{eqnarray*}
\mathcal{R}(v_i \tau_{i+1} v_i) &=& s_{1_n, v_i} \, s_{\overline{v_i}, \tau_{i+1}} \, s_{\overline{v_i \tau_{i+1}}, v_i} = 1_n  \,s_{v_i, \tau_{i+1}} \, 1_n = v_i \gamma_{i+1, i+2}v_i = \gamma_{i, i+2}\\
\mathcal{R}(v_{i+1} \tau_i v_{i+1}) &=& s_{1_n, v_{i+1}} \, s_{\overline{v_{i+1}}, \tau_{i}} \, s_{\overline{v_{i+1} \tau_{i}}, v_{i+1}} = 1_n  \,s_{v_{i+1}, \tau_{i}} \, 1_n = v_{i+1} \gamma_{i, i+1}v_{i+1} = \gamma_{i, i+2}
\end{eqnarray*}

Let us consider now the relation $r_7: \sigma_i \sigma_{i+1} \tau_i = \tau_{i+1} \sigma_i \sigma_{i+1}$ for $VSB_n$. Applying the rewriting process to the left hand side, we obtain the following:
\begin{eqnarray*}
\mathcal{R}(\sigma_i \sigma_{i+1} \tau_i ) &=& s_{1_n, \sigma_i} \, s_{\overline{\sigma_i}, \sigma_{i+1}} \, s_{\overline{\sigma_i \sigma_{i+1} }, \tau_i} = s_{1_n, \sigma_i} \, s_{v_i, \sigma_{i+1}} \, s_{v_i v_{i+1, \tau_i}} \\
 &=& (\mu_{i,i+1}) (v_i \mu_{i+1, i+2} v_i) (v_i v_{i+1}\gamma_{i,i+1} v_{i+1}v_i) \\
 &=& \mu_{i, i+1} \, \mu_{i, i+2} \, \gamma_{i+1, i+2}.
\end{eqnarray*}
For the right hand side we get:
\begin{eqnarray*}
\mathcal{R} (\tau_{i+1} \sigma_i \sigma_{i+1} ) &=& s_{1_n, \tau_{i+1}} \, s_{\overline{\tau_{i+1}}, \sigma_i} \, s_{\overline{\tau_{i+1} \sigma_i}, \sigma_{i+1} }
= s_{1_n, \tau_{i+1}} \, s_{v_{i+1}, \sigma_i} s_{v_{i+1} v_i, \sigma_{i+1}} \\
&=& (\gamma_{i+1, i+2})(v_{i+1} \mu_{i, i+i} v_{i+1}) (v_{i+1} v_i \mu_{i+1, i+2}  v_i v_{i+1}) \\
&=& \gamma_{i+1, i+2} \, \mu_{i, i+2} \, \mu_{i, i+1}.
\end{eqnarray*}

Hence, we arrive at $\mathcal{R}(r_7): \mu_{i, i+1} \, \mu_{i, i+2} \, \gamma_{i+1, i+2} = \gamma_{i+1, i+2} \, \mu_{i, i+2} \, \mu_{i, i+1}$. Conjugating this by all representatives in $\Lambda_n$, relations~\eqref{eq:pureYB-mixt} follow immediately.

Applying the rewriting process $\mathcal{R}$ to the relation $r_7': \sigma_{i+1}\sigma_i \tau_{i+1} = \tau_i \sigma_{i+1}\sigma_i$ we obtain:
\[ \mathcal{R}(r_7'): \gamma_{i, i+1} \, \mu_{i, i+2} \, \mu_{i+1, i+2} = \mu_{i+1, i+2} \, \mu_{i, i+2} \, \gamma_{i, i+1}.\]
Conjugating this by all representatives in $\Lambda_n$, we get relations~\eqref{eq:pureYB-mixt-two}.

By applying the rewriting process to $r_8: \sigma_i \tau_i = \tau_i \sigma_i$, we get the following:
\begin{eqnarray*}
\mathcal{R}(\sigma_i \tau_i ) &=& s_{1_n, \sigma_i} \, s_{\overline{\sigma_i}, \tau_i}  = s_{1_n, \sigma_i} \, s_{v_i, \tau_i} = (\mu_{i,i+1}) (v_i \gamma_{i, i+1}v_i) = \mu_{i,i+1} \, \gamma_{i+1, i}\\
\mathcal{R}(\tau_i \sigma_i  ) &=& s_{1_n, \tau_i} \, s_{\overline{\tau_i}, \sigma_i}  = s_{1_n, \tau_i} \, s_{v_i, \sigma_i} = (\gamma_{i,i+1}) (v_i \mu_{i, i+1}v_i) = \gamma_{i,i+1} \, \mu_{i+1, i}.
\end{eqnarray*}
Therefore, $\mathcal{R}(r_8)$ has the form $\mu_{i,i+1} \, \gamma_{i+1, i} = \gamma_{i,i+1} \, \mu_{i+1, i}$. It follows that relations $\mathcal{R}(\lambda r_8 \lambda^{-1})$ for all representatives $\lambda \in \Lambda_n$ have the form $\mu_{kl} \gamma_{lk} = \gamma_{kl} \mu_{lk}$. Thus relations~\eqref{eq:pureTwist} hold in $VSP_n$.

It remains to show that the rewriting process applied to the commuting relations $r_9$ of $VSB_n$ yields the commuting relations~\eqref{eq:purecommuting} of $VSP_n$. For example, let $|i-j| >1$ and consider the relations $\sigma_i \tau_j = \tau_j \sigma_i$. Then,\begin{eqnarray*}
\mathcal{R}(\sigma_i \tau_j ) &=&s_{1_n, \sigma_i} \, s_{\overline{\sigma_i}, \tau_j}  = (\mu_{i, i+1}) (v_i \gamma_{j, j+1}v_i) = \mu_{i,i+1} \, \gamma_{j, j+1}\\
\mathcal{R}( \tau_j \sigma_i) &=&s_{1_n, \tau_j} \, s_{\overline{\tau_j}, \sigma_i}  = (\gamma_{j, j+1}) (v_j \mu_{i, i+1}v_j) =  \gamma_{j, j+1} \, \mu_{i,i+1}. 
\end{eqnarray*}
We have derived the relation $\mu_{i, i+1} \gamma_{j, j+1} = \gamma_{j, j+1} \mu_{i, i+1}$. Conjugating it by all representatives in $\Lambda_n$, we get:
\[ \mu_{ij} \gamma_{kl} = \gamma_{kl} \mu_{ij}, \,\, \text{where}\,\, \{i, j\} \cap \{ k, l\} = \emptyset. \]
Similarly, by applying the rewriting process to relations $\sigma_i \sigma_j = \sigma_j \sigma_i$ and $\tau_i \tau_j = \tau_j \tau_i$, where $|i-j| >1$, we obtain:
\[ \mu_{ij} \mu_{kl} = \mu_{kl} \mu_{i, j} \,\, \text{and} \,\,\gamma_{ij} \gamma_{kl} = \gamma_{kl} \gamma_{ij}, \,\, \text{where}\,\, \{i, j\} \cap \{ k, l\} = \emptyset.  \] 

Hence, relations~\eqref{eq:purecommuting} hold in $VSP_n$. It is easy to see that the commuting relations $\sigma_i v_j = v_j \sigma_i$ and $\tau_i v_j = v_j \tau_i$ of $VSB_n$ result in the identity relation in $VSP_n$.
 \end{proof}
 
 \begin{remark} 
 We know that $VSP_n$ is a normal submonoid of $VSB_n$. Moreover, $VSB_n = VSP_n  \rtimes_{\phi} S_n$, where $\phi \co S_n \to \text{Aut}(VSP_n)$ is the permutation representation associated with the action of $S_n$ on $VSP_n$ as described by Lemma~\ref{lemma:action}. Therefore, $VSB_n$ splits over $VSP_n$ by $S_n$.
 \end{remark}
 
 
 \section{A representation of $VSB_n$} \label{sec:repres}
 
We wish to provide a  categorical description for the virtual singular braid monoid, $VSB_n$. Since the notion of a monoidal category is the categorification of the notion of a monoid, we define a monoidal category $\textbf{FS}$---the category of fusing strings---that can be used as a basis for studying virtual singular braids and the algebraic structure among them. We also use this category to construct a representation of $VSB_n$. For an introduction to category theory and in particular monoidal categories, we refer the reader to, for example, the books~\cite{BK, E, Ma}. 
 

 \subsection{The category of fusing strings}
$VSB_n$ embeds into $VSB_{n+1}$ via the map $\iota : VSB_n \to VSB_{n+1}$ given by adding a vertical strand to the right of a braid $b \in VSB_n$. This operation results in a braid $\iota(b) \in VSB_{n+1}$, and we can define: \[VSB_{\infty}: =   \bigcup_{n = 1} ^{\infty} VSB_n.\]
 In providing a categorical description of $VSB_{\infty}$, we rely on the definition of the monoid $M_n$ (see Definition~\ref{def:mn}) and the full set of elementary fusing strings. By Theorem~\ref{theorem:iso}, $VSB_n$ and $M_n$ are isomorphic monoids.
 
 \begin{definition}
 The \textit{category $\textbf{FS}$ of fusing strings}  is the strict monoidal category freely generated by one object, denoted by $*$, and the following four morphisms:
 \[ \mu = \raisebox{-20pt}{\includegraphics[height=0.6in]{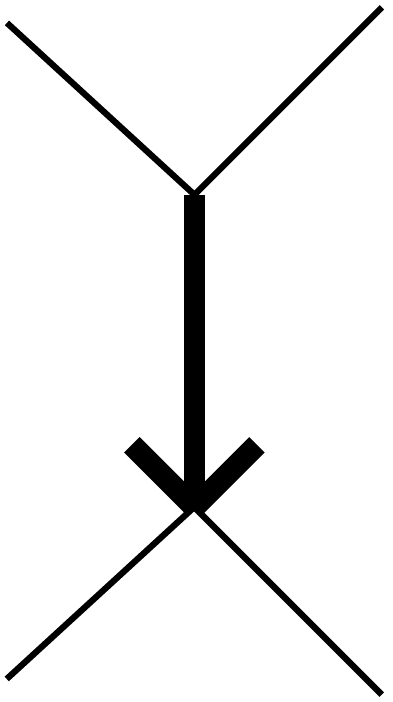}}  \hspace{1cm} 
 \mu^{-1} = \raisebox{25pt}{\includegraphics[height=0.6in, angle = 180]{cat-mu}} \hspace{1cm} 
 \gamma = \raisebox{-20pt}{\includegraphics[height=0.6in]{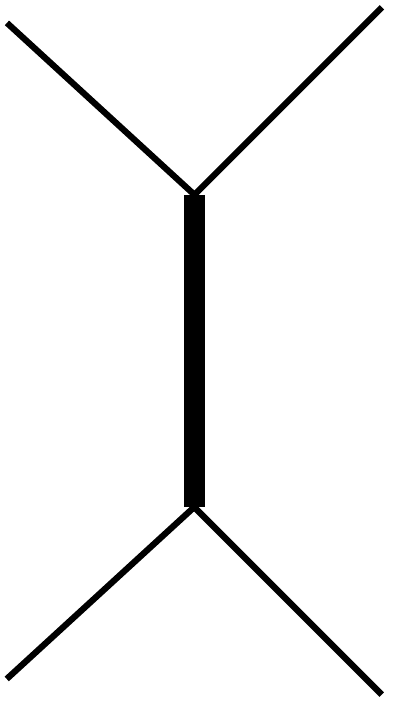}}  \hspace{1cm} 
 v = \raisebox{-20pt}{\includegraphics[height=0.6in]{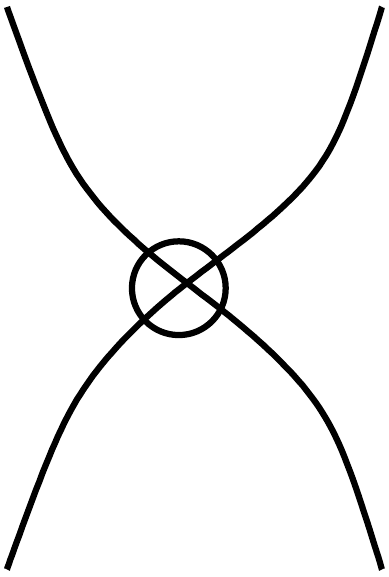}}  \]
where $\mu, \mu^{-1}, \gamma$ and $v : * \otimes * \longrightarrow * \otimes *$.
 \end{definition}
For each non-negative integer $n$, there exists an object $[n]$ in $\textbf{FS}$:
\[ [n] = * \otimes * \otimes \cdots \otimes * \, ,\]
containing $n $ copies of $*$. The object $[n]$ is a categorification of the top or bottom $n$ endpoints of a braid in $VSB_n$. 

A \textit{trivial morphism} in $\textbf{FS}$ is a tensor product of copies of the morphism $\vert: * \longrightarrow *$.
Then, a non-trivial morphism $[n] \longrightarrow [n]$ is a composition of tensor products of trivial morphisms and one of the generating morphisms $\mu, \mu^{-1}, \gamma$ or $v$. We think of a morphism $[n] \longrightarrow [n]$ in $\textbf{FS}$ as a categorification of an element in the monoid $M_n$ and, in particular, of a virtual singular braid in $VSB_n$. For example,
\begin{eqnarray}
v_i &=&  \vert \otimes \cdots \otimes \vert  \otimes v \otimes \vert \otimes \cdots \otimes \vert :  [n] \longrightarrow [n]  \label{categ-v_i} \\
 \gamma_i &=& \vert \otimes \cdots \otimes \vert  \otimes \gamma \otimes \vert \otimes \cdots \otimes \vert : [n] \longrightarrow [n] ,  \label{categ-gamma_i}\\
\mu_i^{\pm 1} &=&  \vert \otimes \cdots \otimes \vert  \otimes \mu^{\pm 1} \otimes \vert \otimes \cdots \otimes \vert :  [n] \longrightarrow [n] \label{categ-mu_i} \\
 \end{eqnarray}
where $v, \gamma$ and $\mu^{\pm 1}$ are in the $i$th place of the tensor product, are categorifications of $v_i$ and the elementary fusing strings $\mu_i$ and $\gamma_i$, respectively.

 \subsection{A representation of $VSB_n$ into linear operators}

Let $V$ be a vector space over a field $\mathbb{K}$ and let $\mathcal{T} \co V \otimes V \to V \otimes V$, $R \co V \otimes V \to V \otimes V$  and $S \co V \otimes V \to V \otimes V$ be linear operators. We define a monoidal functor $\mathcal{F} \co \textbf{FS} \to \mathbb{K}-\textbf{Vec}$ given by $\mathcal{F}(*): =V$ and
\[\mathcal{F}(v): = \mathcal{T}, \,\, \mathcal{F}(\mu): = R,  \,\, \mathcal{F}(\gamma): = S, \]
where $\mathbb{K} -\textbf{Vec}$ is the category of vector spaces over $\mathbb{K}$.
Then, 
\begin{eqnarray*}
\mathcal{F}(v_i) &=& \id_V \otimes \cdots \otimes \id_V \otimes \mathcal{T} \otimes \id_V \cdots \otimes \id_V \\
\mathcal{F}(\mu_i) &=& \id_V \otimes \cdots \otimes \id_V \otimes R \otimes \id_V \cdots \otimes \id_V \\
\mathcal{F}(\gamma_i) &=& \id_V \otimes \cdots \otimes \id_V \otimes S \otimes \id_V \cdots \otimes \id_V ,
\end{eqnarray*}
where $\mathcal{T}, R$ and respectively $S$  occur in the $i$th place of the tensor product. Note that $\mathcal{F}(v_i), \mathcal{F}(\mu_i)$ and $\mathcal{F}(\gamma_i) \in End(V^{\otimes n})$.

\begin{proposition} \label{operators}
If the linear operators $\mathcal{T} \co V \otimes V \to V \otimes V$, $R\co V \otimes V \to V \otimes V$  and $S \co V \otimes V \to V \otimes V$ satisfy the following:  
\begin{enumerate}
\item $ \mathcal{T} \, \text{and} \, R \, \text{are invertible, and} \, \mathcal{T}^2 = \id_{V \otimes V} $
    \vspace{0.2cm}
\item $ (\mathcal{T} \otimes \id_V) (\id_V \otimes \mathcal{T}) (\mathcal{T} \otimes \id_V) = (\id_V \otimes \mathcal{T}) (\mathcal{T} \otimes \id_V)  (\id_V \otimes \mathcal{T}) $ 
     \vspace{0.2cm}
\item $(\mathcal{T} \otimes \id_V) (\id_V \otimes R) (\mathcal{T} \otimes \id_V) = (\id_V \otimes \mathcal{T}) (R \otimes \id_V)  (\id_V \otimes \mathcal{T})$
     \vspace{0.2cm}
\item $(\mathcal{T} \otimes \id_V) (\id_V \otimes S) (\mathcal{T} \otimes \id_V) = (\id_V \otimes \mathcal{T}) (S \otimes \id_V)  (\id_V \otimes \mathcal{T}) $
     \vspace{0.2cm}
\item 
$(\id_V \otimes R) [ (\id_V \otimes \mathcal{T})(R \otimes \id_V) (\id_V \otimes \mathcal{T}) ](R \otimes \id_V) =\\ 
 (R \otimes \id_V) [ (\id_V \otimes \mathcal{T})(R \otimes \id_V) (\id_V \otimes \mathcal{T}) ]  (\id_V \otimes R) 
 $
    \vspace{0.2cm}
\item 
$(\id_V \otimes R) [ (\id_V \otimes \mathcal{T})(R \otimes \id_V) (\id_V \otimes \mathcal{T}) ](S \otimes \id_V) = \\
 (S \otimes \id_V) [ (\id_V \otimes \mathcal{T})(R \otimes \id_V) (\id_V \otimes \mathcal{T}) ]  (\id_V \otimes R) 
$
    \vspace{0.2cm}
\item $R \mathcal{T} S = S \mathcal{T} R$,
\end{enumerate}
then $\mathcal{F}$ is a representation of the $n$-stranded virtual singular braid monoid, $VSB_n$, to a submonoid of $End(V^{\otimes n})$. 
\end{proposition}

\begin{proof} The statement follows from Definition~\ref{def:mn}, Theorem~\ref{theorem:iso}, and the definition of the functor $\mathcal{F}$.
\end{proof}

\begin{remark}
Let $R_{12}: = R \otimes \id_V, \,\,\, R_{23}: = \id_V \otimes R, \,\,\,
 R_{13}: = (\id_V \otimes \mathcal{T})(R \otimes \id_V)(\id_V \otimes \mathcal{T}) $
Then, the fifth equation in Proposition~\ref{operators}, becomes $R_{12} R_{13} R_{23} = R_{23}R_{13}R_{12}$. That is, by the discussion in Section~\ref{sec:YBEincentive}, $R$ satisfies the Yang-Baxter equation.
\end{remark}

\begin{example}
Let $p$ be a prime number and $\xi = e^{2 \pi \sqrt{-1}/p}$. Let $V$ be a vector space over $\mathbb{C}$ with basis $\{b_k \, | \, k \in \mathbb{Z}_p  \}$. We define the linear operators $\mathcal{T} \co V \otimes V \to V \otimes V$, $R \co V \otimes V \to V \otimes V$  and $S \co V \otimes V \to V \otimes V$ given by
\begin{eqnarray}
\mathcal{T}(b_k \otimes b_l) &=& b_l \otimes b_k  \label{eq:v} \\
R(b_k \otimes b_l) &=& \xi^{kl} b_k \otimes b_l  \label{eq:mu}\\
S(b_k \otimes b_l) &=& (\xi^{kl} + \xi^{-kl}) b_k \otimes b_l \label{eq:tau}
\end{eqnarray}

We show that these linear maps satisfy the equations in Proposition~\ref{operators}. It is clear that $\mathcal{T}$ and $R$ are invertible operators and that $\mathcal{T}^2 = \id_{V \otimes V}$. Let $b_k, b_l$ and $b_m$ be arbitrary basis vectors for $V$. Then,
\begin{eqnarray*}
 (\mathcal{T} \otimes \id_V) (\id_V \otimes \mathcal{T}) (\mathcal{T} \otimes \id_V) (b_k \otimes b_l \otimes b_m)&=&  
 (\mathcal{T} \otimes \id_V) (\id_V \otimes \mathcal{T}) (b_l \otimes b_k \otimes b_m)\\
 &=&  (\mathcal{T} \otimes \id_V) (b_l \otimes b_m \otimes b_k)\\
 &=&b_m \otimes b_l \otimes b_k,
\end{eqnarray*}
and
\begin{eqnarray*}
(\id_V \otimes \mathcal{T}) (\mathcal{T} \otimes \id_V)  (\id_V \otimes \mathcal{T}) (b_k \otimes b_l \otimes b_m) &=&
(\id_V \otimes \mathcal{T}) (\mathcal{T} \otimes \id_V) (b_k \otimes b_m \otimes b_l) \\
&=& (\id_V \otimes \mathcal{T}) (b_m \otimes b_k \otimes b_l) \\
&=& b_m \otimes b_l \otimes b_k.
\end{eqnarray*}
Thus, the second equality holds. The third equality follows similarly, as shown below.

\begin{eqnarray*}
(\mathcal{T} \otimes \id_V) (\id_V \otimes R) (\mathcal{T} \otimes \id_V) (b_k \otimes b_l \otimes b_m) &=&
(\mathcal{T} \otimes \id_V) (\id_V \otimes R) (b_l \otimes b_k \otimes b_m)\\
 &=&  (\mathcal{T} \otimes \id_V) (b_l \otimes \xi^{km}b_k \otimes b_m)\\
 &=&\xi^{km}b_k \otimes b_l \otimes b_m,
\end{eqnarray*}
\begin{eqnarray*}
(\id_V \otimes \mathcal{T}) (R \otimes \id_V)  (\id_V \otimes \mathcal{T}) (b_k \otimes b_l \otimes b_m) &=&
(\id_V \otimes \mathcal{T}) (R \otimes \id_V) (b_k \otimes b_m \otimes b_l) \\
&=& (\id_V \otimes \mathcal{T}) (\xi^{km} b_k \otimes b_m \otimes b_l) \\
&=& \xi^{km} b_k \otimes b_l \otimes b_m.
\end{eqnarray*}
For the left hand side of the fourth equality, we obtain:
\begin{eqnarray*}
(\mathcal{T} \otimes \id_V) (\id_V \otimes S) (\mathcal{T} \otimes \id_V) (b_k \otimes b_l \otimes b_m) &=&
(\mathcal{T} \otimes \id_V) (\id_V \otimes S) (b_l \otimes b_k \otimes b_m)\\
 &=&  (\mathcal{T} \otimes \id_V) (b_l \otimes (\xi^{km} + \xi^{-km})b_k \otimes b_m)\\
 &=&(\xi^{km} + \xi^{-km}) b_k \otimes b_l \otimes b_m,
\end{eqnarray*}
and for the right hand side, we obtain:
\begin{eqnarray*}
(\id_V \otimes \mathcal{T}) (S \otimes \id_V)  (\id_V \otimes \mathcal{T}) (b_k \otimes b_l \otimes b_m) &=&
(\id_V \otimes \mathcal{T}) (S \otimes \id_V) (b_k \otimes b_m \otimes b_l) \\
&=& (\id_V \otimes \mathcal{T}) ( (\xi^{km} + \xi^{-km})b_k \otimes b_m \otimes b_l) \\
&=& (\xi^{km} + \xi^{-km}) b_k \otimes b_l \otimes b_m.
\end{eqnarray*}
Therefore, the fourth equality holds. Next, we show that the fifth equality holds.
\begin{eqnarray*}
&& (\id_V \otimes R) [ (\id_V \otimes \mathcal{T})(R \otimes \id_V) (\id_V \otimes \mathcal{T}) ](R \otimes \id_V) (b_k \otimes b_l \otimes b_m) \\
&&=  (\id_V \otimes R) [ (\id_V \otimes \mathcal{T})(R \otimes \id_V)(\id_V \otimes \mathcal{T})] ( \xi^{kl}b_k \otimes b_l \otimes b_m) \\
 &&=  (\id_V \otimes R) [ (\id_V \otimes \mathcal{T})(R \otimes \id_V)] ( \xi^{kl}b_k \otimes b_m \otimes b_l) \\
  &&=  (\id_V \otimes R)  (\id_V \otimes \mathcal{T}) ( \xi^{kl} \xi^{km}b_k \otimes b_m \otimes b_l) \\
  &&=   (\id_V \otimes R) ( \xi^{kl} \xi^{km}b_k \otimes b_l \otimes b_m) \\
    && =  \xi^{kl} \xi^{km} \xi^{lm} b_k \otimes b_l \otimes b_m,
\end{eqnarray*}

\begin{eqnarray*}
&& (R \otimes \id_V) [ (\id_V \otimes \mathcal{T})(R \otimes \id_V) (\id_V \otimes \mathcal{T}) ]  (\id_V \otimes R) 
(b_k \otimes b_l \otimes b_m) \\
&& = (R \otimes \id_V) [ (\id_V \otimes \mathcal{T})(R \otimes \id_V) (\id_V \otimes \mathcal{T}) ]  (b_k \otimes \xi^{lm} b_l \otimes b_m) \\
&&= (R \otimes \id_V) [ (\id_V \otimes \mathcal{T})(R \otimes \id_V) ]   (\xi^{lm}b_k \otimes  b_m \otimes b_l) \\
&&=  (R \otimes \id_V)  (\id_V \otimes \mathcal{T}) ( \xi^{km} \xi^{lm} b_k \otimes  b_m \otimes b_l) \\
&&= (R \otimes \id_V) ( \xi^{km} \xi^{lm} b_k \otimes  b_l \otimes b_m) \\
&&= \xi^{kl} \xi^{km}  \xi^{lm} b_k \otimes  b_l \otimes b_m.
\end{eqnarray*}

The sixth equality is very similar to the one we just verified, and thus we show only some of the steps below.
\begin{eqnarray*}
&& (\id_V \otimes R) [ (\id_V \otimes \mathcal{T})(R \otimes \id_V) (\id_V \otimes \mathcal{T}) ](S \otimes \id_V) (b_k \otimes b_l \otimes b_m) \\
&&=  (\id_V \otimes R) [ (\id_V \otimes \mathcal{T})(R \otimes \id_V)(\id_V \otimes \mathcal{T})] ( (\xi^{kl} + \xi^{-kl})b_k \otimes b_l \otimes b_m) \\
 && =  (\xi^{kl} + \xi^{-kl}) \xi^{km} \xi^{lm} b_k \otimes b_l \otimes b_m,
\end{eqnarray*}

\begin{eqnarray*}
&& (S \otimes \id_V) [ (\id_V \otimes \mathcal{T})(R \otimes \id_V) (\id_V \otimes \mathcal{T}) ]  (\id_V \otimes R) 
(b_k \otimes b_l \otimes b_m) \\
&&= (S \otimes \id_V) ( \xi^{km} \xi^{lm} b_k \otimes  b_l \otimes b_m) \\
&&=  (\xi^{kl} + \xi^{-kl})  \xi^{km}  \xi^{lm} b_k \otimes  b_l \otimes b_m.
\end{eqnarray*}

Finally, we show that the seventh equality holds.
\begin{eqnarray*}
R \mathcal{T} S (b_k \otimes b_l ) &=&
 R \mathcal{T} ( (\xi^{kl} + \xi^{-kl}) b_k \otimes b_l  ) \\
 &=& R ( (\xi^{kl} + \xi^{-kl})  b_l \otimes b_k ) \\
 &=& \xi^{kl}(\xi^{kl} + \xi^{-kl})  b_l \otimes b_k, 
\end{eqnarray*}
\begin{eqnarray*}
S \mathcal{T} R (b_k \otimes b_l ) &=& S \mathcal{T}  ( \xi^{kl} b_k \otimes b_l ) \\
&=&  S  ( \xi^{kl} b_l \otimes b_k ) \\
&=&  \xi^{kl}(\xi^{kl} + \xi^{-kl})  b_l \otimes b_k.
\end{eqnarray*}
Therefore, the operators $\mathcal{T}, R$ and $S$ defined by equations~\eqref{eq:v} --~\eqref{eq:tau} satisfy the properties given in Proposition~\ref{operators}, and therefore they provide an example of a functor $\mathcal{F}$ that yields a representation of $VSB_n$ to a submonoid of $End(V^{\otimes n})$. Moreover, the operator $R$ satisfies the Yang-Baxter equation.
\end{example}
 
 \textit{Final remarks.} We can quotient the set $End_{\textbf{FS}}([n])$ of morphisms $[n] \to [n]$ in $\textbf{FS}$ by the finite set of relations obtained by substituting the expressions for the morphisms $v_i, \gamma_i$ and $\mu_i^{\pm 1}$ as defined by equations~\eqref{categ-v_i}--\eqref{categ-mu_i} into the relations listed in Proposition~\ref{prop:redvssn}. We use $\textbf{FS}_{/{r}}$ to denote the resulting quotient category ($r$ stands for `relations'). Note that $\textbf{FS}_{/{r}}$ is still a strict monoidal category. We chose to make use of the relations in Proposition~\ref{prop:redvssn} versus those in Definition~\ref{def:mn}, so that we have a minimal set of generators for the monoid of all morphisms $[n] \to [n]$ in $\textbf{FS}_{/{r}}$.
 It is easy to see that $VSB_n$ is isomorphic to the monoid of morphisms $[n] \to [n]$ in $\textbf{FS}_{/{r}}$. Then one can work with this quotient category as a playground for studying virtual singular braids and the algebraic structure of the $n$-stranded virtual singular braid monoid.
 \bigbreak
 
 \noindent \textbf{Acknowledgements.} The first named author was partially supported by a grant from the Simons Foundation ($\#355640$, Carmen Caprau).


\end{document}